\documentclass[12pt]{amsart}

\usepackage{stmaryrd}
\usepackage{mathtools}
\def\multichoose#1#2{\ensuremath{\left(\kern-.3em\left(\genfrac{}{}{0pt}{}{#1}{#2}\right)\kern-.3em\right)}}
\def\mbinom#1#2{\ensuremath{\left(\kern-.3em\left(\genfrac{}{}{0pt}{}{#1}{#2}\right)\kern-.3em\right)}}

\newcommand{\llb}[0]{\llbracket}
\newcommand{\rrb}[0]{\rrbracket}

\usepackage{graphicx}
\usepackage{color}
\usepackage{etoolbox}
\usepackage[margin=1in]{geometry}
\usepackage{amsmath,amsthm,amssymb,tikz,tikz-cd}
\usepackage{hyperref}
\hypersetup{
    colorlinks = true,
    citecolor = orange,%
}
\usepackage[noabbrev]{cleveref}
\usepackage{subcaption}

\makeatletter
\patchcmd{\@settitle}{\uppercasenonmath\@title}{}{}{}
\patchcmd{\@setauthors}{\MakeUppercase}{}{}{}
\patchcmd{\section}{\scshape}{}{}{}
\makeatother
\makeatletter
\patchcmd{\@maketitle}
  {\ifx\@empty\@dedicatory}
  {\ifx\@empty\@date \else {\vskip2ex 
  \centering\footnotesize\@date\par\vskip1ex}\fi
   \ifx\@empty\@dedicatory}
  {}{}
\patchcmd{\@adminfootnotes}
  {\ifx\@empty\@date\else \@footnotetext{\@setdate}\fi}
  {}{}{}
\makeatother

\newcommand{\pres}[1]{\prescript{#1}{}}

\theoremstyle{plain}
\newtheorem{theorem}{Theorem}[section]

\newtheorem{lemma}[theorem]{Lemma}
\newtheorem{prop}[theorem]{Proposition}
\newtheorem{conj}[theorem]{Conjecture}
\newtheorem{question}[theorem]{Question}
\newtheorem{corollary}[theorem]{Corollary}

\theoremstyle{definition}
\newtheorem{remark}[theorem]{Remark}
\newtheorem{example}[theorem]{Example}
\newtheorem{definition}[theorem]{Definition}

\newcommand{\calG}[0]{\mathcal{G}}
\newcommand{\wt}{\mathrm{wt}}

\title
[Higher dimer covers on snake graphs]
{\large Higher Dimer Covers on Snake Graphs}

\author[G. Musiker]{Gregg Musiker$^\clubsuit$}
\author[N. Ovenhouse]{Nicholas Ovenhouse$^\heartsuit$}
\author[R. Schiffler]{Ralf Schiffler$^\spadesuit$}
\author[S. W. Zhang]{Sylvester W. Zhang$^\diamondsuit$}

\thanks{$^\clubsuit$\href{mailto:musiker@umn.edu}{musiker@umn.edu} University of Minnesota.}

\thanks{$^\heartsuit$\href{mailto:ovenhou3@msu.edu}{ovenhou3@msu.edu} Michigan State University}
\thanks{$^\spadesuit$\href{mailto:schiffler@math.uconn.edu}{schiffler@math.uconn.edu} University of Connecticut}
\thanks{$^\diamondsuit$\href{mailto:sylvesterzhang@math.ucla.edu}{sylvesterzhang@math.ucla.edu} UCLA}

  

\begin{document}

\begin{abstract}
Snake graphs are a class of planar graphs that are important in the theory of cluster algebras. Indeed, the Laurent expansions of the cluster variables in cluster algebras from surfaces are given as weight generating functions for 1-dimer covers (or perfect matchings) of snake graphs. 
Moreover, the enumeration of 1-dimer covers of snake graphs provides a combinatorial interpretation of continued fractions. In particular, the number of 1-dimer covers of the snake graph $\mathcal{G}[a_1,\ldots,a_n]$ is the numerator of the continued fraction $[a_1,\ldots,a_n]$. This number is equal to the top left entry of the matrix product $\left(\begin{smallmatrix} a_1&1\\1&0 \end{smallmatrix}\right)
\cdots \left(\begin{smallmatrix} a_n&1\\1&0 \end{smallmatrix}\right)$.

In this paper, we give enumerative results on $m$-dimer covers of snake graphs. We show that the number of $m$-dimer covers of the snake graph $\mathcal{G}[a_1,\ldots,a_n]$ is the top left entry of a product of analogous $(m+1)$-by-$(m+1)$ matrices. 
We discuss how our enumerative results are related to other known combinatorial formulas,
and we suggest a generalization of continued fractions based on our methods. 
These generalized continued fractions provide some interesting open questions and
a possibly novel approach towards Hermite's problem for cubic irrationals.
\end{abstract}

\maketitle

\setcounter{tocdepth}{1}
\tableofcontents

\section{Introduction}
Snake graphs first appeared in the theory of cluster algebras. For a cluster algebra from a marked surface, 
the papers \cite{ms10,msw_11} established a combinatorial formula for the Laurent expansion of a cluster variable as a weighted sum 
over all dimer covers (perfect matchings) of an associated snake graph. In \cite{msw13}, snake graphs were used to  construct canonical bases for the cluster algebra.
Later it was shown in \cite{cs_18} that there is a bijection between continued fractions $[a_1,a_2,\ldots,a_n] $ and snake graphs $\mathcal{G}[a_1,a_2,\ldots,a_n]$ 
such that the number of dimer covers of the snake graph equals the numerator of the continued fraction. See Theorem~\ref{thm1} for a precise statement.
Moreover, this equation of integers can be lifted to the cluster algebra expressing the cluster variables as (numerators of) continued fractions of 
Laurent monomials \cite{cs_18,rab}. In \cite{bms,ls_19} this connection was used to express the Jones and the Alexander polynomial 
of 2-bridge knots as specializations  of cluster variables, and in \cite{llrs} to study the order relation on Markov numbers.

In \cite{moz1, moz2} it was shown that the double dimer covers of snake graphs arise combinatorially in super cluster algebras. 
Indeed, the $\lambda$-lengths on the decorated super Teichm\"{u}ller space, introduced in \cite{pz_19}, were shown to be weighted sums 
over double dimer covers of the associated snake graph. Also, in \cite{moz3}, a flat $\mathrm{Osp}(1|2)$ graph connection was constructed
which represents the associated point in the decorated super Teichm\"{u}ller space, and the entries of the holonomy matrices of this connection
were shown to also be weighted sums of double dimer covers on snake graphs.

It is therefore natural to ask if there is an $m$-version of Theorem~\ref{thm1}, that is, can we use the continued fraction description 
of the snake graph $\calG=\calG[a_1,a_2,\ldots,a_n]$ to count the number of $m$-dimer covers of $\calG$?
An \emph{$m$-dimer cover} of a graph is a multiset of edges so that each vertex is incident to $m$ edges.
This question is one of the open problems in algebraic combinatorics posed at the OPAC conference \cite{opac} by the third author,
and that conference was where the work on this project began.

In this paper, we solve this problem completely. 
Indeed, it is well-known that the continued fraction $[a_1,a_2,\ldots,a_n]=\frac{p_n}{q_n}$ can be computed using a product of 2-by-2 matrices 
\begin{equation} \label{eq:contfrac} \begin{pmatrix} a_1 & 1 \\ 1 & 0 \end{pmatrix}
 \begin{pmatrix} a_2 & 1 \\ 1 & 0 \end{pmatrix}
 \ldots
  \begin{pmatrix} a_n & 1 \\ 1 & 0 \end{pmatrix} = \begin{pmatrix} p_n& p_{n-1} \\ q_n & q_{n-1} \end{pmatrix}.
\end{equation}
The integers $a_i$ can be interpreted as the number of 1-dimer covers of the snake graph $\calG[a_i]$. This snake graph consists of $a_i-1$ tiles that are glued together in a zigzag pattern. Let us denote the number of $m$-dimer covers of this snake graph by $\llb a_i\rrb_m$. 
It is straightforward to verify that the number of $1$-dimer covers of the zig-zag snake graph $\calG[a_i]$ is precisely $a_i$. Thus the above matrix product becomes 
\begin{equation}\begin{pmatrix} \llb a_1\rrb _1 & 1 \\ 1 & 0 \end{pmatrix}
 \begin{pmatrix} \llb a_2\rrb _1 & 1 \\ 1 & 0 \end{pmatrix}
 \ldots
  \begin{pmatrix} \llb a_n\rrb _1 & 1 \\ 1 & 0 \end{pmatrix} = \begin{pmatrix} p_n& p_{n-1} \\ q_n & q_{n-1} \end{pmatrix}.
\end{equation}
We generalize this matrix product by replacing each 2-by-2 matrix $\left(\begin{smallmatrix} a & 1 \\ 1 & 0 \end{smallmatrix}\right)
$ by the matrix $\Lambda^{(m)}(a)$ of size $(m+1)$-by-$(m+1)$ in which each ascending skew diagonal is constant and the entries on the skew diagonals from top left to bottom right are \[\llb a\rrb _m,\llb a\rrb _{m-1},\ldots ,\llb a\rrb _2,\llb a\rrb _1,1,0,\ldots, 0.\] 
Equivalently, the $(i,j)$-th entry of $\Lambda^{(m)}(a)$ is $\llb a\rrb_{m+2-i-j}$ where we set $\llb a\rrb_0=1$ and $\llb a\rrb_k = 0$ when $k < 0$.
For example, for $m=3$, we have 
\[\Lambda^{(3)}(a)= \begin{pmatrix} 
\llb a \rrb_3 &\llb a \rrb_2&\llb a \rrb_1& 1 \\
\llb a \rrb_2 &\llb a \rrb_1&1&0\\
\llb a \rrb_1&1&0&0\\
1&0&0&0
 \end{pmatrix}.
 \]
 Our first main result is the following.
 
\begin{theorem}\label{thm intro1}
The number of $m$-dimer covers of the snake graph $\calG[a_1,a_2,\ldots,a_n]$ is equal to the top left entry of the matrix product $\Lambda^{(m)}(a_1)\Lambda^{(m)}(a_2)\cdots\Lambda^{(m)}(a_n)$.
\end{theorem}

We also give dimer theoretic interpretations of all other entries of this matrix product in Theorem~\ref{thm:matrix} {as well as an expression of this top left entry using a straight snake graph with weighted edges in Theorem~\ref{thm 1}}.  {In \Cref{sect 5}, we investigate the poset structure of $m$-dimer covers and related generating functions.}

Pushing the analogy with the $m=1$ case even further, we use these products of $\Lambda$-matrices to define a seemingly new
higher-dimensional generalization of continued fractions. We define $\mathrm{CF}_m(a_1,\dots,a_n) := (r_{m,m}, r_{m-1,m}, \dots, r_{1,m}, r_{0,m})^\top$
to be the first column of the matrix $\Lambda^{(m)}(a_1)\Lambda^{(m)}(a_2) \cdots \Lambda^{(m)}(a_n)$, normalized so that $r_{0,m} = 1$. Equivalently, we consider
$\mathrm{CF}_m$ as the {homogeneous coordinates of a point in $\Bbb{QP}^m$}. This is discussed in \Cref{sec:fixing_m}.

Since every rational number $x \geq 1$ is represented by a finite continued fraction given by a sequence of positive integers $x = [a_1,a_2,\dots,a_n]$,
we get a family of maps $r_{i,m} \colon \Bbb{Q}_{\geq 1} \to \Bbb{Q}_{\geq 1}$ for any $1 \leq i \leq m$. Our first main result about these generalized
continued fractions is that they extend to irrationals as well, so that $r_{i,m}(x)$ is well-defined for all real $x \geq 1$.

\begin {theorem}\label{thm:intro2}
    For any sequence of positive integers $a_1,a_2,a_3,\dots$, and any $1 \leq i \leq m$, the sequence $r_{i,m}(a_1,\dots,a_n)$ converges as $n \to \infty$.
\end {theorem}

We also re-interpret Theorem \ref{thm intro1} in terms of these generalized continued fractions to obtain a generalization of Theorem \ref{thm1} (originally from \cite{cs_18}): 
Let $\Omega_m(\mathcal{G})$ denote the set of $m$-dimer covers of the graph $\mathcal{G}$.

\begin {theorem}\label{thm:intro3}
    Let $x \in \Bbb{Q}_{\geq 1}$ be a rational number with finite continued fraction $x = [a_1,a_2,\dots,a_n]$. Then
    \[ r_{m,m}(x) = \frac{\# \Omega_m(\mathcal{G}[a_1,\dots,a_n])}{\# \Omega_m(\mathcal{G}[a_2,\dots,a_n])}. \]
\end {theorem}

As an application, we suggest a possible new approach to Hermite's problem. It is well-known that the continued fraction of a number is
eventually periodic if and only if the number is a quadratic irrational. Hermite asked if it is possible to represent numbers by a sequence of integers
so that the sequence is eventually periodic if and only if the number is a cubic irrational.
In our setting, the cubic case corresponds to double dimer covers, i.e. $m=2$.
In Theorem \ref{thm:eventually_periodic}, we show that for an eventually periodic sequence,
the values of $r_{i,m}$ lie in a number field which is a degree $(m+1)$ extension of $\Bbb{Q}$.

Based on numerical experiments, we conjecture that the maps $r_{i,m}$ are monotone increasing, continuous, and map $\Bbb{R}_{\geq 1}$ bijectively onto itself.
See \Cref{sec:future_directions}, and in particular \Cref{fig:r2_graph}. If one could find an algorithm for inverting the map $r_{2,2}$ (which we are presently not able to do),
this would provide a way to see whether cubic irrationals come from eventually periodic sequences.

\section{Recollections} In this section we recall the concepts and results we need.

\subsection{Snake graphs} In this subsection, we recall the definition of a snake graph following the exposition in \cite{cs_18}. 
An illustration is given in Figure~\ref{fig snakegraph}.

 A \emph{tile} is a planar graph with four vertices and four edges that has the shape of a square. A {\em snake graph} $\mathcal{G}$ is a connected planar graph consisting of a finite sequence of tiles $G_1,G_2,\ldots, G_d$ with $d \geq 1,$ such that
$G_i$ and $G_{i+1}$ share exactly one edge $e_i$ and this edge is either the north edge of $G_i$ and the south edge of $G_{i+1}$ or the east edge of $G_i$ and the west edge of $G_{i+1}$,  for each $i=1,\dots,d-1$.
A snake graph $\mathcal{G}$ is called {\em straight} if all its tiles lie in one column or one row, and a snake graph is called {\em zigzag} if no three consecutive tiles are straight.
 We say that two snake graphs are \emph{isomorphic} if they are isomorphic as graphs.

\begin{figure}
\begin{center}
\begingroup%
  \makeatletter%
  \providecommand\color[2][]{%
    \errmessage{(Inkscape) Color is used for the text in Inkscape, but the package 'color.sty' is not loaded}%
    \renewcommand\color[2][]{}%
  }%
  \providecommand\transparent[1]{%
    \errmessage{(Inkscape) Transparency is used (non-zero) for the text in Inkscape, but the package 'transparent.sty' is not loaded}%
    \renewcommand\transparent[1]{}%
  }%
  \providecommand\rotatebox[2]{#2}%
  \newcommand*\fsize{\dimexpr\f@size pt\relax}%
  \newcommand*\lineheight[1]{\fontsize{\fsize}{#1\fsize}\selectfont}%
  \ifx\svgwidth\undefined%
    \setlength{\unitlength}{245.00282764bp}%
    \ifx\svgscale\undefined%
      \relax%
    \else%
      \setlength{\unitlength}{\unitlength * \real{\svgscale}}%
    \fi%
  \else%
    \setlength{\unitlength}{\svgwidth}%
  \fi%
  \global\let\svgwidth\undefined%
  \global\let\svgscale\undefined%
  \makeatother%
  \begin{picture}(1,0.32096568)%
    \lineheight{1}%
    \setlength\tabcolsep{0pt}%
    \put(0,0){\includegraphics[width=\unitlength,page=1]{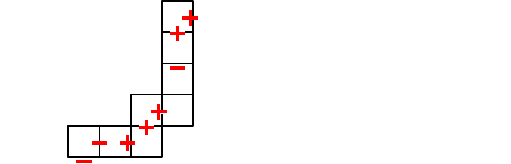}}%
    \put(0.00153059,0.04988186){\color[rgb]{0,0,0}\makebox(0,0)[lt]{\lineheight{0}\smash{\begin{tabular}[t]{l} \end{tabular}}}}%
    \put(0,0){\includegraphics[width=\unitlength,page=2]{figures/snakegraph.pdf}}%
  \end{picture}%
\endgroup%

 \caption{The left hand picture shows the snake graph of the continued fraction $[2,4,1,2]$. The sign sequence is given in red.  The sign changes occur in the tiles $G_{\ell_i}$, with $\ell_i=2,6,7$.
 The picture on the right  shows the dual snake graph $(\calG[2,4,1,2])^*=\calG[1,2,1,1,3,1]$.} \label{fig snakegraph}
\end{center}
\end{figure}

A {\em sign function} $f$ on a snake graph $\mathcal{G}$ is a map $f$ from the set of edges of $\mathcal{G}$ to the set $\{ +,- \}$ 
such that on every tile in $\mathcal{G}$ the north and the west edge have the same sign, the south and the east edge have the same sign 
and the sign on the north edge is opposite to the sign on the south edge. 
Equivalently, the signs are constant along lines of slope 1, and the signs on two adjacents diagonals must be different.
The 
snake graph $\mathcal{G}$ is determined by a sequence of tiles $G_1,\ldots,G_d$ and a sign function $f$ on the interior edges $e_1,\ldots,e_{d-1}$ of $\mathcal{G}$.   Denote by $e_0$ the south edge of the first tile and choose an edge $e_d$ to be either the north edge or the east edge of the last tile $G_d$. Then we obtain a sign sequence 
\begin{equation}
 \label{seq}
 (f(e_0) , f(e_1) ,\ldots ,f(e_{d-1}), f(e_d)).
\end{equation}

\medskip

Now let \[ [a_1,a_2,\ldots,a_n]= a_1+\cfrac{1}{a_2+\cfrac{1}{\ddots +\cfrac{1}{a_n}}}\]
 be a continued fraction with all $a_i$ positive integers, and let $d= a_1+a_2+\cdots +a_n -1$.
Consider the following sign sequence
\begin{equation}
 \label{eqsign} 
\begin{array}{cccccccc}
  ( \underbrace{ - ,\ldots,- },&  \underbrace{  +,\ldots,+ },&  \underbrace{ - ,\ldots,- },& \ldots,&  \underbrace{\pm ,\ldots,\pm }) .  \\
 a_1 & a_2 & a_3&\ldots&a_n
\end{array} 
\end{equation}

 Thus each integer $a_i$ corresponds to a maximal subsequence of constant sign in the sequence (\ref{eqsign}). 
We let $\ell_i$ denote  the position of the last term in the $i$-th subsequence, thus  
$\ell_i = \sum_{j=1}^{i} a_{j}.$

\begin{definition}\label{def:snake_for_continued_fraction}
	 The snake graph
 $\mathcal{G}[a_1,a_2,\ldots,a_n] $ of the continued fraction $[a_1,a_2,\ldots,a_n]$ is the snake graph with $d$ tiles determined by the sign sequence (\ref{eqsign}). 
\end{definition}

The following theorem is shown in \cite{cs_18}. It gives a combinatorial realization of continued fractions as quotients of cardinalities of sets.
\begin{theorem}\cite[Theorem 3.4]{cs_18}\label{thm1} 
 If  $\Omega_1(\calG)$ denotes the set of perfect matchings of $\calG$ then
 \[ [a_1,a_2,\ldots,a_n] =\frac{\#\Omega_1(\calG [a_1,a_2,\ldots,a_n])}{\#\Omega_1(\calG[a_2,\ldots,a_n])},\]
 and the right hand side is a reduced fraction. 
\end{theorem}

\begin{example}
 The snake graph in the top row of Figure \ref{fig mlatticepath} has continued fraction $[1,1,1,1] = \frac{5}{3}$. The numerator $5$ is equal to the number of perfect matchings as shown in the figure. The denominator 3 corresponds to the number of perfect matchings of the subgraph given by the top two tiles. These are given by the restriction of the matchings $P_1,P_2,P_3$ to that subgraph. 
\end{example}

\subsection{$m$-dimer covers and $m$-lattice paths} \label{sec:mdimer}
Let $G$ be a graph. A \emph{dimer cover} (or 1-dimer cover) of $G$ is the same as a perfect matching of $G$, thus it is a subset $P$ of the set of edges of $G$ such that, for every vertex $x$ of $G$, there exists exactly one edge in $P$ that is incident to $x$. 

If $m$ is a positive integer, an $m$-dimer cover of $G$ is a multiset $P$ whose elements are edges of $G$ such that, for every vertex $x$ of $G$, there exist exactly $m$ edges in $P$ that are incident to $x$. 

\begin{remark}
    An $m$-dimer cover is not the same as an $m$-tuple of $1$-dimer covers. 
    For example a single tile has two 1-dimer covers and three 2-dimer covers (and not four). 
    It is also not the same as an $m$-multiset of 1-dimer covers. For example, the snake graph $\calG[1,1,1,1]$ has five 1-dimer covers shown in the top row of Figure~\ref{fig mlatticepath}, and hence 15 2-multisets of 1-dimer covers. However, it only has 14 2-dimer covers, because the multisets $\{P_2,P_5\}$ and $\{P_3,P_4\}$ give rise to the same 2-dimer cover.
\end{remark}

A \emph{lattice path} in a snake graph $\calG$ is a sequence of vertices  $(a_0,b_0),(a_1,b_1),\ldots, (a_l,b_l)$ of $\calG$ 
such that each step $(a_{i+1},b_{i+1})-(a_i,b_i)$ is an edge of $\calG$, directed either north or east, 
and $(a_0,b_0)$ is the south-west vertex of the first tile and $(a_l,b_l)$ is the north-east vertex of the last tile.  

An \emph{$m$-multipath} on $\calG$ is a multiset with $m$ elements, each of which is a lattice path on $\calG$. 
The \emph{edge multiset} of an $m$-multipath is the multiset whose elements are the edges of the lattice paths in the multiset.  
Two $m$-multipaths are said to be equivalent if they have the same edge multiset.
An \emph{$m$-lattice path} in a snake graph $\calG$ is the equivalence class of an $m$-multipath.

To illustrate these concepts we show the five lattice paths $w_1,\ldots,w_5$ of the snake graph $\calG[4]$ on the left hand side in Figure~\ref{fig mlatticepath}. These five paths give rise to  the 15 $2$-multipaths $\{w_i,w_j\}$ with $1\le i\le j\le5$. Among these, the two multipaths $\{w_2,w_5\}$ and $\{w_3,w_4\}$ have the same edge multiset, which is shown in the right picture of the figure. Thus by definition, $\{w_2,w_5\}$ and $\{w_3,w_4\}$  define the same 2-lattice path. The total number of 2-lattice paths is  14. 
\begin{figure}
\centering
\newcommand{\svgwidth}{\textwidth}
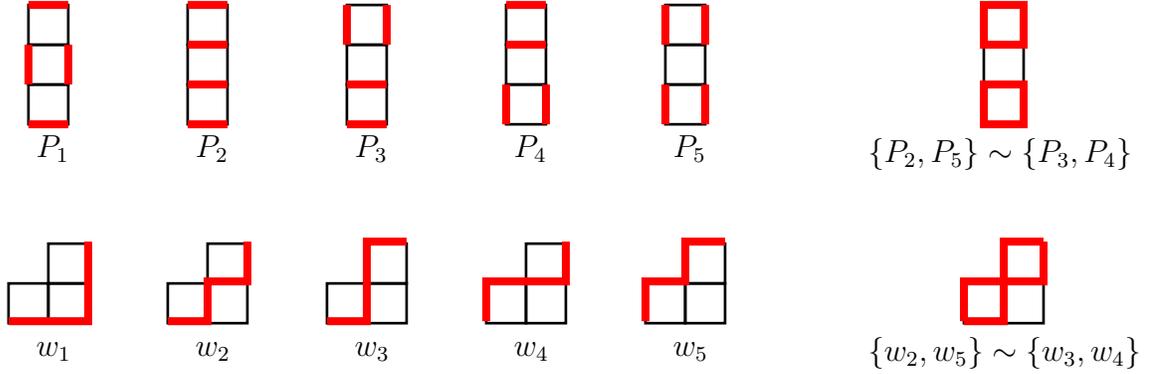
\caption{The pictures in the top row show the five 1-dimer covers $P_1,\ldots,P_5$ of the snake graph $\calG[1,1,1,1]$. 
The pictures in the bottom row show the five lattice paths $w_1,\ldots,w_5$ of $\calG[4]$, which is the dual of $\calG[1,1,1,1]$. 
The top picture on the right hand side shows that the multisets $\{P_2,P_5\}$ and $\{P_3,P_4\}$ produce the same 2-dimer cover, 
and the picture below shows that the multipaths $\{w_2,w_5\}$ and $\{w_3,w_4\}$ produce the same 2-lattice path.}
\label{fig mlatticepath}
\end{figure}

\subsection{The dual snake graph}
Following \cite{claussen20,propp05}, we associate a dual snake graph $\calG^*$ to every snake graph $\calG$ as follows. If $\calG=\calG[a_1,a_2,\ldots,a_n]$, then $\calG^*=\calG[b_1,b_2,\ldots,b_m]$, where the sequence $(b_i)$ is obtained from the sequence $(a_i)$ by the following two steps. 
\begin{enumerate}
\item Replace each $a_i$ by $1+1+\cdots+1$, the sum of $a_i$ ones.
\item Replace each $+$ sign by a comma and each comma by a $+$ sign.
\end{enumerate}

Continuing  the example in Figure \ref{fig snakegraph}, the dual of $\calG[2,4,1,2]$ is $\calG[1,2,1,1,3,1]$, because step (1) produces the sequence 
$1+1,1+1+1+1,1,1+1$
and then continuing with step (2) produces
$1,1+1,1,1,1+1+1,1=1,2,1,1,3,1$.

\begin{remark}\label{rem dual G}
(a) Three consecutive tiles in $\calG$ form a straight subsnake graph if and only if the corresponding tiles form a zigzag subsnake graph in $\calG^*$. Moreover, the fact that $a_1\ge 2$ if and only if $b_1=1$ implies that the first step in $\calG$ is to the right if and only if the first step in $\calG^*$ is up.

(b) In the dual snake graph $\mathcal{G}^*[b_1,b_2,\ldots,b_m]$ each $b_i$ corresponds to a straight segment and each comma corresponds to a corner. 
    The lengths of the segments are indicated by the $b_i$'s.

(c) A snake graph and its dual have the same number of tiles, because $a_1+\cdots+a_n=b_1+\cdots +b_m$. 

(d) The values of the continued fractions $[a_1,\ldots,a_n]$ and $[b_1,\ldots,b_m]$ are not the same. In our example, we have $[2,4,1,2]=\frac{31}{14}$ and $[1,2,1,1,3,1]=\frac{32}{23}$. Thus the duality induces a non-trivial involution $D$ on $\mathbb{Q}$, and, when passing to infinite continued fractions, a non-trivial involution on $\mathbb{R}$. 
This involution was studied in \cite{au_15}.
\end{remark}

The reason why the dual snake graph is important to us is the following result  \cite[Theorem 6.22]{claussen20}. It is illustrated in Figure~\ref{fig mlatticepath}. 
\begin{theorem}
 \label{thm clausen}
 There is an order preserving bijection between the lattice of perfect matchings of $\calG$ and the lattice of lattice paths on $\calG^*$.
\end{theorem}
{
Here we give a brief illustration of the bijection of Theorem \ref{thm clausen} (see \Cref{fig dimer-to-path}) and refer to \cite{claussen20} for details. 
Label each diagonal of the tiles of $\calG$ by $d_i$ from bottom-left to top-right. For an $m$-dimer cover $M$ (or $m$-lattice path) of a snake graph $\calG$, 
define the operations $\tau_i$ to be reflecting every edge above $d_i$ about the anti-diagonal. Note that this operation changes both the graph and the dimer cover on it. 
Then iteratively applying $\tau_{1},\tau_2,\dots$ at all squares will take an $m$-dimer cover of $\calG$ to an $m$-lattice path of $\calG^*$. 
}

\begin{figure}
\centering
\begin{tikzpicture}[scale = 0.55]
    \node [] (0) at (-9, -2) {};
    \node [] (1) at (-8, -2) {};
    \node [] (2) at (-9, -1) {};
    \node [] (3) at (-8, -1) {};
    \node [] (4) at (-9, 0) {};
    \node [] (5) at (-8, 0) {};
    \node [] (6) at (-9, 1) {};
    \node [] (7) at (-8, 1) {};
    \node [] (8) at (-5.5, -2) {};
    \node [] (9) at (-4.5, -2) {};
    \node [] (10) at (-5.5, -1) {};
    \node [] (11) at (-4.5, -1) {};
    \node [] (12) at (-3.5, -1) {};
    \node [] (13) at (-3.5, -2) {};
    \node [] (14) at (-2.5, -2) {};
    \node [] (15) at (-2.5, -1) {};
    \node [] (16) at (0, -2) {};
    \node [] (17) at (1, -2) {};
    \node [] (18) at (2, -2) {};
    \node [] (19) at (2, -1) {};
    \node [] (20) at (1, -1) {};
    \node [] (21) at (0, -1) {};
    \node [] (22) at (1, 0) {};
    \node [] (23) at (2, 0) {};
    \node [] (24) at (-7, 0) {};
    \node [] (25) at (-7, 1) {};
    \node [] (26) at (-3.5, 0) {};
    \node [] (27) at (-2.5, 0) {};
    \node [] (28) at (3, 0) {};
    \node [] (29) at (3, -1) {};
    \node [] (30) at (5.5, -2) {};
    \node [] (31) at (6.5, -2) {};
    \node [] (32) at (7.5, -2) {};
    \node [] (33) at (7.5, -1) {};
    \node [] (34) at (6.5, -1) {};
    \node [] (35) at (5.5, -1) {};
    \node [] (36) at (6.5, 0) {};
    \node [] (37) at (7.5, 0) {};
    \node [] (38) at (7.5, 1) {};
    \node [] (39) at (6.5, 1) {};
    \node [] (40) at (10, -2) {};
    \node [] (41) at (11, -2) {};
    \node [] (42) at (12, -2) {};
    \node [] (43) at (12, -1) {};
    \node [] (44) at (11, -1) {};
    \node [] (45) at (10, -1) {};
    \node [] (46) at (11, 0) {};
    \node [] (47) at (12, 0) {};
    \node [] (48) at (12, 1) {};
    \node [] (49) at (11, 1) {};
    \node [] (50) at (-7.25, -1) {};
    \node [] (51) at (-6.25, -1) {};
    \node [] (52) at (-1.75, -1) {};
    \node [] (53) at (-0.75, -1) {};
    \node [] (54) at (3.75, -1) {};
    \node [] (55) at (4.75, -1) {};
    \node [] (56) at (8.25, -1) {};
    \node [] (57) at (9.25, -1) {};

    \draw [ultra thick, color = red] (2.center) to (0.center);
    \draw [ultra thick, color = red] (3.center) to (1.center);
    \draw [ultra thick, color = red] (4.center) to (5.center);
    \draw [ultra thick, color = red] (6.center) to (7.center);
    \draw (6.center) to (4.center);
    \draw (4.center) to (2.center);
    \draw (2.center) to (3.center);
    \draw (3.center) to (5.center);
    \draw (5.center) to (7.center);
    \draw (0.center) to (1.center);
    \draw [densely dotted] (2.center) to node [fill=white, inner sep = 0.1mm] {\scriptsize $d_1$} (1.center);
    \draw [ultra thick, color = red] (10.center) to (8.center);
    \draw [ultra thick, color = red] (10.center) to (11.center);
    \draw [ultra thick, color = red] (12.center) to (13.center);
    \draw [ultra thick, color = red] (15.center) to (14.center);
    \draw (8.center) to (9.center);
    \draw (9.center) to (11.center);
    \draw (11.center) to (12.center);
    \draw (12.center) to (15.center);
    \draw (9.center) to (13.center);
    \draw (13.center) to (14.center);
    \draw [densely dotted] (11.center) to node [fill=white, inner sep = 0.1mm] {\scriptsize $d_2$} (13.center);
    \draw [ultra thick, color = red] (25.center) to (24.center);
    \draw (7.center) to (25.center);
    \draw (5.center) to (24.center);
    \draw (26.center) to (12.center);
    \draw (27.center) to (15.center);
    \draw [ultra thick, color = red] (26.center) to (27.center);
    \draw [ultra thick, color = red] (21.center) to (16.center);
    \draw [ultra thick, color = red] (21.center) to (20.center);
    \draw [ultra thick, color = red] (20.center) to (19.center);
    \draw [ultra thick, color = red] (22.center) to (23.center);
    \draw [ultra thick, color = red] (28.center) to (29.center);
    \draw [densely dotted] (22.center) to node [fill=white, inner sep = 0.1mm] {\scriptsize $d_3$} (19.center);
    \draw (16.center) to (18.center);
    \draw (18.center) to (23.center);
    \draw (23.center) to (28.center);
    \draw (29.center) to (19.center);
    \draw (22.center) to (17.center);
    \draw [ultra thick, color = red] (35.center) to (30.center);
    \draw [ultra thick, color = red] (35.center) to (34.center);
    \draw [ultra thick, color = red] (34.center) to (33.center);
    \draw (30.center) to (32.center);
    \draw (36.center) to (31.center);
    \draw [ultra thick, color = red] (37.center) to (33.center);
    \draw [ultra thick, color = red] (39.center) to (38.center);
    \draw [densely dotted] (39.center) to node [fill=white, inner sep = 0.1mm] {\scriptsize $d_4$} (37.center);
    \draw (38.center) to (37.center);
    \draw (33.center) to (32.center);
    \draw (36.center) to (39.center);
    \draw (36.center) to (37.center);
    \draw [ultra thick, color = red] (45.center) to (40.center);
    \draw [ultra thick, color = red] (45.center) to (44.center);
    \draw [ultra thick, color = red] (44.center) to (43.center);
    \draw (40.center) to (42.center);
    \draw (46.center) to (41.center);
    \draw [ultra thick, color = red] (47.center) to (43.center);
    \draw (43.center) to (42.center);
    \draw (46.center) to (49.center);
    \draw (46.center) to (47.center);
    \draw [ultra thick, color = red] (48.center) to (47.center);
    \draw (49.center) to (48.center);
    \draw [->] (50.center) to node [above] {$\tau_1$} (51.center) ;
    \draw [->](52.center) to node [above] {$\tau_2$}(53.center);
    \draw [->](54.center) to node [above] {$\tau_3$}(55.center);
    \draw [->](56.center) to node [above] {$\tau_4$}(57.center);
\end{tikzpicture}
\caption{Bijection between dimer covers of $\calG$ and lattice paths of $\calG^*$.}
\label{fig dimer-to-path}
\end{figure}
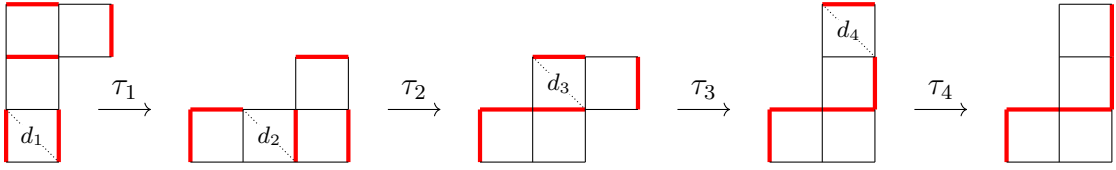

\section{Matrix Formula for $m$-Dimer Covers on Snake Graphs}

In this section, we introduce the notation $\llb a_1,\dots,a_n\rrb_m$ as a weighted sum of $m$-dimer covers on a straight snake graph 
and realize it as the top left entry of a matrix product $\Lambda^{(m)}(a_1)\cdots\Lambda^{(m)}(a_n)$. In section~\ref{sect 4}, we will show that it
is equal to the number of $m$-dimer covers of the snake graph $\calG[a_1,\ldots,a_n]$. 

Let $\Omega_m(G)$ denote the set of all $m$-dimer covers of a graph $G$ and $\Omega_m[a_1,\dots,a_n]$ denote the set of $m$-dimer covers on the snake graph $\mathcal G[a_1,\dots,a_n]$.

Also, define $\mathcal{L}_m[a_1,\dots,a_n]$ to be the set of $m$-lattice paths on the snake graph $\mathcal{G}[a_1,\dots,a_n]$,
and let $\mathcal{L}^*_m[a_1,\dots,a_n]$ be the set of $m$-lattice paths on the dual snake graph $\mathcal{G}^*$.
The following statement was noted by Propp, and is also explained in Claussen's thesis.

\begin {prop} \label{prop:dual_snake_bijection}
\textup{(\cite[Section~4]{propp05} and \cite[Chapter~6]{claussen20}).}
    There is a bijection between $\Omega_m[a_1,\dots,a_n]$ and $\mathcal{L}^*_m[a_1,\dots,a_n]$.
\end {prop}

\begin {lemma} \label{lem:straight_snake}
    For a positive integer $a$, we have
    \[ \#\Omega_m[a] = \#\mathcal{L}^*_m[a] = \mbinom{a}{m} = {a+m-1\choose m} \]
\end {lemma}
\begin{proof}
    The graph $\mathcal{G}^*[a]$ is a vertical column of $a-1$ tiles. Any lattice path must go right on one of the horizontal edges,
    of which there are $a$. So  an $m$-lattice path amounts to choosing $m$ of these $a$ horizontal steps (possibly with repetition).
    There are thus $\mbinom{a}{m}$ choices. \end{proof}
\begin{remark}\label{rm:dimer-vs-path}
	Note that this fact can be proven directly working exclusively in terms of $m$-dimer covers of such a snake graph, but the proof is more complicated.
\end{remark}

Let $\mathcal G[1^n]=\mathcal G[\underbrace{1,\dots,1}_{n\text{ times}}]$ be the straight snake graph whose vertical edges are denoted $e_1,\dots,e_n$,
\[
    \mathcal G[{1^n}]=\begin{tikzpicture}[baseline=1em]
	\draw (1,0)--(0,0)-- node[left]{$ e_1$}(0,1) -- (1,1) -- node [left] {$e_2$} (1,0)--(2,0)-- node [left] {$e_3$} (2,1)--(1,1);
	\draw [dashed](2,0)--(4,0);
	\draw [dashed](2,1)--(4,1);
	\draw (4,0)--(4,1)--(5,1)--node[right]{$e_n$}(5,0)--(4,0);
	\end{tikzpicture}.
\]

In Definition \ref{def:straight-snake} below, we introduce a notation for the \emph{weighted} enumeration of the $m$-dimer covers on this graph.
We will later see that the \emph{unweighted} enumeration of $m$-dimer covers on an arbitrary snake graph can be written in terms of this weighted
sum on $\mathcal{G}[1^n]$.

\begin{definition}\label{def:straight-snake}
We define $$\llb a_1,\dots,a_n\rrb_m :=\sum_{p\in\Omega_m[1^n]}\prod_{i=1}^{n}\wt_p(e_i)$$
where
\[\wt_p(e_i) = \mbinom{a_i}{k} = {a_i+k-1\choose k}\]if $e_i$ is covered by $k$ edges in $p$, and denote $p(e_i)=k$. \end{definition}

The dual of $\mathcal{G}[1,1,\dots,1]$ is the zigzag snake graph 
\begin {center}
\begin {tikzpicture}
    \draw (0,0) -- (2,0) -- (2,1) -- (3,1) -- (3,2) -- (4,2) -- (4,3) -- (2,3) -- (2,2) -- (1,2) -- (1,1) -- (0,1) -- cycle;
    \draw (1,0) -- (1,1) -- (2,1) -- (2,2) -- (3,2) -- (3,3);

    \draw (0.5,0) node[below] {$e_1$};
    \draw (0.8,0.5) node  {$e_2$};
    \draw (1.5,0.8) node  {$e_3$};
    \draw (1.8,1.5) node  {$e_4$};
    \draw (2.5,1.7) node  {\reflectbox{$\ddots$}};
    \draw (2.9,2.5) node  {$e_{n-1}$};
    \draw (3.5,3.2) node  {$e_n$};
\end {tikzpicture}
\end {center}

So we can equivalently define
\[ \llb a_1,\dots,a_n\rrb_m := \sum_{p \in \mathcal{L}^*_m[1^n]} \prod_{i=1}^{n} \wt_p(e_i) \]

\begin{definition} \label{def:lambda_matrix}
	Define the matrix $\Lambda^{(m)}(a):=(A_{ij})_{1\leq i,j,\leq m+1}$ where $A_{ij}=\#\Omega_{{m+2-i-j}}[a]= \llb a\rrb_{m+2-i-j}$. Recall that  $\llb a\rrb_0=1$ and $\llb a\rrb_k=0$ if $k<0$. For example
	\[
            \Lambda^{(4)}(a):=
            \begin{pmatrix}
                    \llb a\rrb_4&\llb a\rrb_3&\llb a\rrb_2&\llb a\rrb_1&1\\
                    \llb a\rrb_3&\llb a\rrb_2&\llb a\rrb_1&1&0\\
                    \llb a\rrb_2&\llb a\rrb_1&1&0&0\\
                    \llb a\rrb_1&1&0&0&0\\
                    1&0&0&0&0
            \end{pmatrix} 
            = \begin{pmatrix}
                \mbinom{a}{4} & \mbinom{a}{3} & \mbinom{a}{2} & a & 1 \\[0.5ex]
                \mbinom{a}{3} & \mbinom{a}{2} & a             & 1 & 0 \\[0.5ex]
                \mbinom{a}{2} & a             & 1             & 0 & 0 \\
                a             & 1             & 0             & 0 & 0 \\
                1             & 0             & 0             & 0 & 0
            \end{pmatrix}
	\]
\end{definition}

\begin{definition}
	We define $\pres{i}\llb a_1,\dots,a_n \rrb^j_m$ to be the weighted sum of the set of $m$-dimer covers on the straight snake graph $\mathcal G[1^n]$ 
        such that the left most edge is covered by $i$ edges and the right most edge is covered by $j$ edges, or equivalently,
	\[\pres{i}\llb a_1,\dots,a_n \rrb^j_m:=
	\sum_{
	\substack{p\in \Omega_m \mathcal G[1^n] \\ p(e_1)=i,\ p(e_n)=j} }\prod_{i=1}^n \wt_p(e_i)\]
	where $\wt_p(e_i)$ is as defined in Definition \ref{def:straight-snake}.
\end{definition}

\begin {prop}\label{propprop 3.7}
    For a sequence $a_1,\dots,a_n$, we have 
    \[ \pres{m}\llb 1, a_1, \dots, a_n, 1 \rrb^m_m = \llb a_1, \dots, a_n \rrb_m \]
\end {prop}
\begin {proof}
    The terms in the sum $\pres{m}\llb 1, a_1, \dots, a_n, 1 \rrb^m_m$ are indexed by $m$-dimer covers on $\mathcal{G}[1^{n+2}]$,
    where the left-most and right-most edges (which are labelled by $1$'s) have multiplicity exactly $m$. Removing these multiplicity-$m$
    edges, what remains is an $m$-dimer cover on $\mathcal{G}[1^n]$ (the subgraph without the first and last square). Since the internal vertical edges
    are labelled by $a_1, \dots, a_n$, this gives a weight-preserving bijection.
\end {proof}

We are now ready to state our matrix formula.
\begin{theorem}\label{thm:matrix}
	Let $X=(X_{ij})_{1\leq i,j\leq m+1}$ be the matrix product \(X=\Lambda^{(m)}(a_1)\cdots \Lambda^{(m)}(a_n)\). Then
	 $X_{i,j}=\pres{m+1-i}\llb 1,a_1,\dots,a_n,1\rrb^{m+1-j}_m$.\footnote{Note that $\llb 1,a_1,\dots,a_n,1\rrb$ is very similar to the extended triangulation in section 5 of \cite{moz3}.} 
	
	In particular, the top-left entry is 
        $$X_{1,1}=\pres{m}\llb 1,a_1,\dots,a_n,1\rrb^{m}_m=\llb a_1,\dots,a_n\rrb_m.$$
\end{theorem}

Once Theorem \ref{thm:matrix} is proven, we will show our main result (Theorem \ref{thm intro1}) that this quantity also equals the number 
of $m$-dimer covers in Section \ref{sect 4} after introducing more background on working with these matrices.

We will need the following lemma to prove Theorem \ref{thm:matrix}. 
\begin{lemma}\label{lem:matrix-prod}For any positive integers $a_1,\dots,a_n$, we have
		\[\pres{i}\llb 1,a_1,\dots,a_n,1\rrb_m^{j}=\sum_{r=0}^{j}\pres{i}\llb 1,a_1,\dots,a_{n-1},1\rrb_m^{r+m-j} \cdot \llb a_n\rrb_{r}\]
\end{lemma}
\begin{proof}
Let $\mathcal G[1^{n+2}]$ be the straight snake graph with $n+1$ tiles where the vertical edges are labelled by $e_1,\dots,e_{n+2}$. 
We consider the set of all $m$-dimer covers such that the left most edge is covered by an $i$-dimer and the right most edge is covered by a $j$-dimer, 
and re-group them based on $p(e_{n+1})$, as illustrated below.

\[
\begin{tikzpicture}[baseline=1.25em,scale=0.5]
	\draw (0,0)--(2,0)--(4,0)--(8,0)--(10,0);
	\draw (0,2)--(2,2)--(4,2)--(8,2)--(10,2);
	\draw (0,0)--(0,2);
	\draw (2,0)--(2,2);
	\draw (6,0)--(6,2);
	\node () at (4,1) {$\cdots$};
	\draw (8,0)--(8,2);
	\draw (10,0)--(10,2);
	
	\node () [fill=white,scale=0.6] at (0,1) {$i$};
	\node () [fill=white,scale=0.6] at (10,1) {$j$};
\end{tikzpicture}=\sum_{r=0}^{j}
\begin{tikzpicture}[baseline=1.25em,scale=0.5]
	\draw (0,0)--(2,0)--(4,0)--(8,0)--(10,0);
	\draw (0,2)--(2,2)--(4,2)--(8,2)--(10,2);
	\draw (0,0)--(0,2);
	\draw (2,0)--(2,2);
	\draw (6,0)--(6,2);
	\draw (8,0)--(8,2);
	\draw (10,0)--(10,2);
	\node () at (4,1) {$\cdots$};
	
	\node () [fill=white,scale=0.6] at (0,1) {$i$};
	\node () [fill=white,scale=0.6] at (10,1) {$j$};
	\node () [fill=white,scale=0.6] at (8,1) {$r$};
	\node () [fill=white,scale=0.6] at (9,2.4) {$m-j$};
	\node () [fill=white,scale=0.6] at (9,-0.4) {$m-j$};
	\node () [fill=white,scale=0.6] at (7,2.4) {$j-r$};
	\node () [fill=white,scale=0.6] at (7,-0.4) {$j-r$};
\end{tikzpicture}
\]

Now we analyze the terms on the right hand side summation. Each term corresponds to a set of $m$-dimer covers $p$ of $\mathcal G[1^{n+2}]$ 
with $p(e_1)=i,p(e_{n+1})=r,p(e_{n+2})=j$. Observe that in this case the horizontal edges in the second to last tile are forced to be covered by $(j-r)$ edges.

Take a smaller snake graph $\mathcal G[1^{n+1}]$, and to distinguish from the previous one, we label the vertical edges by $f_1,\dots,f_{n+1}$. 
Now consider the set of $m$-dimer covers $p$ of $\mathcal G[1^{n+1}]$ such that $p(f_1)=i,p(f_{n+1})=m-j+r$, 
which corresponds to the terms in $\pres{i}\llb 1,a_1,\dots,a_{n-1},1\rrb_m^{m-j+r}$. 
Notice that the horizontal edges of the last tile are forced to be covered by $(j-r)$ edges. 
There is an obvious bijection $\varphi$ between the above mentioned two sets given by
\(\varphi(p)(f_k)=p(e_k)\) for $1\leq k\leq n$.

\[
\begin{tikzpicture}[baseline=1.25em,scale=0.5]
	\draw (0,0)--(2,0)--(4,0)--(8,0)--(10,0);
	\draw (0,2)--(2,2)--(4,2)--(8,2)--(10,2);
	\draw (0,0)--(0,2);
	\draw (2,0)--(2,2);
	\draw (6,0)--(6,2);
	\draw (8,0)--(8,2);
	\draw (10,0)--(10,2);
	\node () at (4,1) {$\cdots$};
	\node [scale=0.5,]at (0.25,0.45) {$\textcolor{blue}{e_1}$};
	\node [scale=0.5,]at (2.25,0.45) {$\textcolor{blue}{e_2}$};
	\node [scale=0.5,]at (6.25,0.45) {$\textcolor{blue}{e_{n}}$};
	\node [scale=0.5,]at (8.45,0.45) {$\textcolor{blue}{e_{n+1}}$};
	\node [scale=0.5,]at (10.45,0.45) {$\textcolor{blue}{e_{n+2}}$};
	
	\node () [fill=white,scale=0.7] at (0,1) {$i$};
	\node () [fill=white,scale=0.7] at (10,1) {$j$};
	\node () [fill=white,scale=0.7] at (8,1) {$r$};
	\node () [fill=white,scale=0.7] at (9,2.4) {$m-j$};
	\node () [fill=white,scale=0.7] at (9,-0.4) {$m-j$};
	\node () [fill=white,scale=0.7] at (7,2.4) {$j-r$};
	\node () [fill=white,scale=0.7] at (7,-0.4) {$j-r$};
\end{tikzpicture}\;
\begin{tikzpicture}
\draw (0,0.8mm) -- (0,-0.8mm);
\newlength\mylength
\setlength{\mylength}{\widthof{$\hspace{0.6em}\varphi\hspace{0.6em}$}}
\draw[->] (0,0) -- (1.2\mylength,0) node[above,midway] {$\varphi$};
\end{tikzpicture}\;
\begin{tikzpicture}[baseline=1.25em,scale=0.5]
	\draw (0,0)--(2,0)--(4,0)--(8,0);
	\draw [dotted] (8,0) --(10,0);
	\draw (0,2)--(2,2)--(4,2)--(8,2);
	\draw [dotted] (8,2)--(10,2);
	\draw (0,0)--(0,2);
	\draw (2,0)--(2,2);
	\draw (6,0)--(6,2);
	\draw (8,0)--(8,2);
	\draw [dotted] (10,0)--(10,2);
	\node () at (4,1) {$\cdots$};
	
	\node () [fill=white,scale=0.7] at (0,1) {$i$};
	\node () [fill=white,scale=0.7] at (8,1) {$m-j+r$};
	\node () [fill=white,scale=0.7] at (7,2.4) {$j-r$};
	\node () [fill=white,scale=0.7] at (7,-0.4) {$j-r$};
	
	\node [scale=0.5,]at (0.25,0.45) {$\textcolor{blue}{f_1}$};
	\node [scale=0.5,]at (2.25,0.45) {$\textcolor{blue}{f_2}$};
	\node [scale=0.5,]at (6.25,0.45) {$\textcolor{blue}{f_{n}}$};
	\node [scale=0.5,]at (8.45,0.45) {$\textcolor{blue}{f_{n+1}}$};

\end{tikzpicture}
\]

Let $\wt_p(e_k)$ and $\wt_p(f_k)$ be defined as in $\llb 1,a_1,\dots,a_{n},1\rrb_m$ and $\llb 1,a_1,\dots,a_{n-1},1\rrb_m$ respectively. 
Notice that we have $\wt_p(e_k)=\wt_{\varphi(p)}(f_k)$ for all $1\leq k\leq n$ by construction. In addition, $\wt_{\varphi(p)}(f_{n+1})=1$,
because the last entry in the continued fraction is equal to 1, and $\mbinom{1}{k} = 1$ for any $k$.

Now take any $p\in \Omega_m[1^{n+2}]$ with $p(e_1)=i,p(e_{n+1})=r$ and $p(e_{n+2})=j$. Then we have
\[\prod_{k=1}^{n+2}\wt_p(e_k)=\left(\prod_{k=1}^{n}\wt_p(e_k)\right)\wt_p(e_{n+1})\wt_p(e_{n+2})=\left(\prod_{k=1}^{n}\wt_p(e_k)\right)\llb a_n\rrb_r\]
since $\wt_p(e_{n+1})=\llb a_n\rrb_r$ and $\wt_p(e_{n+2}) = \llb 1 \rrb_j = 1$. Under the above mentioned bijection, the perfect matching $\varphi(p)$ contributes 
\[\prod_{k=1}^{n+1}\wt_{{\varphi(p)}}(f_k)=\prod_{k=1}^{n}\wt_{\varphi(p)}(f_k)=\prod_{k=1}^n\wt_p(e_k)\]
to $\pres{i}\llb 1,a_1,\dots,a_{n-1},1\rrb_m^{m+j-r}$. Putting these equations together we get
\[\prod_{k=1}^{n+2}\wt_p(e_k)=\left(\prod_{k=1}^{n+1}\wt_{{\varphi(p)}}(f_k)\right)\llb a_n\rrb_r  
\] 
and summing over all $p$'s gives
\[\pres{i}\llb 1,a_1,\dots,a_n,1\rrb_m^{j}=\sum_{r=0}^{j}\pres{i}\llb 1,a_1,\dots,a_{n-1},1\rrb_m^{r+m-j} \cdot \llb a_n\rrb_{r}\]
as desired.
\end{proof}

For convenience, we state a different version of the lemma.
\begin{corollary}
	\label{cor:matrix-prod}For any positive integers $a_1,\dots,a_n$, we have
		\[\pres{m+1-i}\llb 1,a_1,\dots,a_n,1\rrb_m^{m+1-j}=\sum_{r=0}^{m+1-j}\pres{m+1-i}\llb 1,a_1,\dots,a_{n-1},1\rrb_m^{r+j-1} \cdot \llb a_n\rrb_{r}\]

\end{corollary}
\begin{proof}
	It follows from Lemma \ref{lem:matrix-prod} by setting $i\mapsto m+1-i$ and $j\mapsto m+1-j$.
\end{proof}

\begin{proof}[Proof of Theorem \ref{thm:matrix}]
    We proceed by induction on $n$, noting that the base case of $\Lambda^{(m)}(a_1)_{i,j} =$ $~^{m+1-i}\llb 1,a_1,1 \rrb^{m+1-j}_m = \llb a_1 \rrb_{m+2-i-j}$ 
    follows from Proposition \ref{propprop 3.7} and Definition \ref{def:lambda_matrix}.
     
    Let $\Lambda^{(m)}(a_1)\cdots \Lambda^{(m)}(a_{n-1})=(X_{ij})=X$ and let $X\Lambda^{(m)}(a_{n})=(Y_{ij})$. Suppose the formula holds for $X_{ij}$ and we will prove it for $Y_{ij}$.
    
    Matrix multiplication yields \[Y_{ij}=\sum_{k=1}^{m+1}X_{ik}\llb a_n\rrb_{m+2-k-j}\]

    By induction, $X_{ik} = \pres{m+1-i}\llb 1,a_1,\dots,a_{n-1},1\rrb_m^{m+1-k}$, so we have
    \[Y_{ij} = \sum_{k=1}^{m+1}\pres{m+1-i}\llb 1,a_1,\dots,a_{n-1},1\rrb_m^{m+1-k} \cdot \llb a_n\rrb_{m+2-j-k}\]
    If we re-index so that $r = m+2-j-k$, we get
    \[Y_{ij} = \sum_{r=1-j}^{m+1-j}\pres{m+1-i}\llb 1,a_1,\dots,a_{n-1},1\rrb_m^{r+j-1} \cdot \llb a_n\rrb_{r}\]
    Since $j \geq 1$, the first few terms of this sum may have $r < 0$. Those terms will automatically be zero,
    since they have a factor of $\llb a_n \rrb_r =  \mbinom{a_n}{r}$, which is zero for negative $r$.
    Removing these terms that are automatically zero, this becomes
    \[Y_{ij} = \sum_{r=0}^{m+1-j}\pres{m+1-i}\llb 1,a_1,\dots,a_{n-1},1\rrb_m^{r+j-1} \cdot \llb a_n\rrb_{r}\]
    By Corollary \ref{cor:matrix-prod}, we get that $Y_{ij} = \pres{m+1-i}\llb 1,a_1,\dots,a_n,1\rrb_m^{m+1-j}$.
\end{proof}

\section {Matrix Products in $\mathrm{SL}_{m+1}(\Bbb{Z})$}\label{sect 4}
In this section, we prove Theorem \ref{thm intro1}.
Let us consider the following integer matrices.
\[\Lambda(a)=  \begin{pmatrix} a & 1 \\ 1 & 0 \end{pmatrix} \quad\quad L = \begin{pmatrix} 1 & 0 \\ 1 & 1 \end{pmatrix} \quad \text{ and } \quad R = \begin{pmatrix} 1 & 1 \\ 0 & 1 \end{pmatrix} \]
The matrices $L,R$ generate the group $\mathrm{SL}_2(\Bbb{Z})$, and $\Lambda(a)\in \mathrm{GL}_2(\Bbb{Z})$.
It is well-known that the continued fraction $[a_1,a_2,\ldots,a_n]$
can be computed by the matrix product

\begin{equation}\label{eq mat}
 \Lambda(a_1)\Lambda(a_2)\cdots\Lambda(a_n) = \begin{pmatrix} p_n& p_{n-1} \\ q_n & q_{n-1} \end{pmatrix}
\end{equation}

where $\frac{p_n}{q_n}=[a_1,a_2,\ldots,a_n]$ and 
$\frac{p_{n-1}}{q_{n-1}}=[a_1,a_2,\ldots,a_{n-1}]$ (with the convention that $\frac{p_0}{q_0} = \frac{1}{0}$). See for example \cite{reutenauer}. 

Moreover, if $W = \begin{pmatrix} 0 & 1 \\ 1 & 0 \end{pmatrix}$, then
\[ W \Lambda(a) =\begin{pmatrix} 1&0\\a&1 \end{pmatrix}
=\begin{pmatrix} 1&0\\1&1 \end{pmatrix}^a =L^a
\quad \text{ and } \quad 
 \Lambda(a) W =\begin{pmatrix} 1&a\\0&1 \end{pmatrix}
=\begin{pmatrix} 1&1\\0&1\end{pmatrix}^a =R^a.
\]
Since $W^2=1$, we can rewrite equation~(\ref{eq mat}) as follows
\[
\begin{pmatrix} p_n& p_{n-1} \\ q_n & q_{n-1} \end{pmatrix}
=
\left\{
\begin{array}{ll}
R^{a_1}L^{a_2}R^{a_3}L^{a_4}\cdots R^{a_{n-1}}L^{a_n} &\textup{if $n$ is even;}\\
R^{a_1}L^{a_2}R^{a_3}L^{a_4}\cdots L^{a_{n-1}}R^{a_n}W &\textup{if $n$ is odd.}\\
\end{array}
\right.
\]
%

\begin {example}
	The continued fraction for $\frac{43}{30}$ is $[1,2,3,4]$, and we have
	\[ RL^2R^3L^4 = \Lambda(1)\Lambda(2)\Lambda(3)\Lambda(4) = \begin{pmatrix} 43 & 10 \\ 30 & 7 \end{pmatrix} \]
\end {example}

\begin {definition}
	Let $L_m$ and $R_m$ be the following lower and upper triangular matrices in $\mathrm{SL}_{m+1}(\Bbb{Z})$ whose non-zero entries are all equal to $1$:
	\[
		L_m = \begin{pmatrix}
			1 & 0 & 0 & \cdots & 0 \\
			1 & 1 & 0 & \cdots & 0 \\
			\vdots & \vdots & \vdots & \ddots & \vdots \\
			1 & 1 & 1 & \cdots & 1
		\end{pmatrix} 
		\quad \text{ and } \quad
		R_m = \begin{pmatrix}
			1 & 1 & 1 & \cdots & 1 \\
			0 & 1 & 1 & \cdots & 1 \\
			\vdots & \vdots & \vdots & \ddots & \vdots \\
			0 & 0 & 0 & \cdots & 1
		\end{pmatrix}
	\]
\end {definition}

\begin {remark}
	The $2 \times 2$ matrices $L$ and $R$ described earlier are the special cases when $m=1$. We will usually omit the subscript $m$,
	and just use ``$L$'' and ``$R$'', even when $m > 1$, when no confusion will arise.
\end {remark}

\begin {lemma}\label{lemlem 44}
	If $k \geq 0$, then the $(i,i+k)$-entry of $R^a$ is given by
	\[ (R^a)_{i,i+k} = \mbinom{a}{k} \]
        All entries below the main diagonal in $R^a$ are zero.
\end {lemma}
\begin {proof}
        The fact that $R^a$ is upper-triangular is clear (since $R$ is upper-triangular).

        To show the formula for the entries above the diagonal, 
	induct on $a$. When $a=1$, we always have $\mbinom{1}{k} = \binom{1+k-1}{k} = \binom{k}{k} = 1$. And indeed, $R^1 = R$ has all entries above the diagonal equal to 1.
	For $a > 1$, and $k \geq 0$, we have
	\begin {align*} 
		(R^a)_{i,i+k}   &= (R \cdot R^{a-1})_{i,i+k} \\
				&= \sum_\ell R_{i \ell} (R^{a-1})_{\ell, i+k} \\
				&= \sum_{\ell \geq i} (R^{a-1})_{\ell, i+k} \\
				&= \sum_{\ell \geq i} \mbinom{a-1}{i+k-\ell} \\
				&= \sum_{p = 0}^k \mbinom{a-1}{p}
	\end {align*}
	Recall that $\mbinom{a-1}{p} = \binom{a+p-2}{p} = \binom{a+p-2}{a-2}$. Summing over $p$ gives $\mbinom{a}{k}$, which follows from a well-known formula 
	for the sum of a diagonal of Pascal's triangle.
\end {proof}

\begin {remark}
	Since $L = R^\top$, Lemma \ref{lemlem 44} also gives an expression for powers of $L$:
        \[ (L^a)_{i+k,i} = \mbinom{a}{k} \]
\end {remark}

Let $W$ be the permutation matrix for the longest word in the symmetric group $S_{m+1}$. 
It is the $(m+1) \times (m+1)$ anti-diagonal matrix with $1$'s along the anti-diagonal.

Then the matrix $\Lambda^{(m)}(a)$ from Definition \ref{def:lambda_matrix} is equal to
\begin {equation} \label{eqn:lambda_LR}
    \Lambda^{(m)}(a) = W L^a = R^a W 
\end {equation}

Again, when $m=1$ this agrees with the $2 \times 2$ matrix $\Lambda$ mentioned above. The following gives another interpretation of Theorem \ref{thm:matrix}.

\begin {theorem}
	The matrix $X$ from Theorem \ref{thm:matrix} is equal to
	\[ 
            X = \begin{cases}
                R^{a_1} L^{a_2} R^{a_3} L^{a_4} \cdots L^{a_n} & \text{ if $n$ is even;} \\
                R^{a_1} L^{a_2} R^{a_3} L^{a_4} \cdots R^{a_n} W & \text{ if $n$ is odd.}
            \end{cases}
        \]
        \qed
\end {theorem}

Our next goal is to show that the number $ \llb a_1, \dots, a_n \rrb_m $ defined in Definition~\ref{def:straight-snake} and computed in Theorem~\ref{thm:matrix} is equal to the number of $m$-dimer covers of the snake graph $\calG[a_1,\ldots,a_n]$. This will then complete the proof of Theorem~\ref{thm intro1}.
\begin {lemma} \label{lem:bracket_equivalence}
    Let $a_1,\dots,a_n$ be a sequence of positive integers.
    \begin {enumerate}
        \item[$(a)$] If $a_n > 1$, then $\llb a_1,\dots, a_n \rrb_m = \llb a_1, \dots, a_n - 1, 1 \rrb_m$.
        \item[$(b)$] If $a_1 > 1$, then $\llb a_1,\dots, a_n \rrb_m = \llb 1, a_1-1, \dots, a_n \rrb_m$.
    \end {enumerate}
\end {lemma}
\begin {proof}
    $(a)$ Let $X = \Lambda^{(m)}(a_1) \cdots \Lambda^{(m)}(a_n)$, and $X' = \Lambda^{(m)}(a_1) \cdots \Lambda^{(m)}(a_n-1)\Lambda^{(m)}(1)$. If $n$ is even, then
    $X = R^{a_1}L^{a_2} \cdots R^{a_{n-1}}L^{a_n}$, and $X' = X L^{-1}RW$. 
        Note that $L^{-1}$ is the lower triangular unipotent matrix
        with $-1$'s just below the diagonal (in the $(i,i-1)$-positions), and zeros elsewhere. 
    So the first column of $L^{-1}RW$ is $e_1$, and so
    the first columns of $X$ and $X'$ are the same. In particular, the $(1,1)$-entries are the same.

    \medskip

    If $n$ is odd, then $X = R^{a_1}L^{a_2} \cdots R^{a_n}W$, and $X' = X WR^{-1}L$, and again the first column is preserved.

    $(b)$ Let $X$ be as above, and this time $X' = \Lambda^{(m)}(1)\Lambda^{(m)}(a_1-1) \Lambda^{(m)}(a_2) \cdots \Lambda^{(m)}(a_n)$.
    Note that the $\Lambda$-matrices are symmetric, we have $X^\top=\Lambda^{(m)}(a_n) \Lambda^{(m)}(a_{n-1}) \cdots \Lambda^{(m)}(a_1)$.
    Therefore $(X')^\top = X^\top L^{-1}RW$ (if $n$ is even), which has the same first column as $X^\top$ by the argument above.
    Taking transpose, this means $X'$ and $X$ have same first row. The calculation when $n$ is odd is similar.
\end {proof}

\begin{remark} \label{rem:allri}
We note for later use that since we proved the entire first column of $X$ and of $X'$ are the same, not only do we get the equality of the $(1,1)$-entries, as explicitly stated in the $n$ even 
case for statement (a), but we in fact get equality of $X$ and $X'$ for all entries $(i,1)$ in the first column. For part (b), the same holds for the first row of the matrices.
\end{remark}

In the proof of the following theorem, we choose to use lattice paths of the dual snake graph $\mathcal{G}^*$ (rather than dimer covers of $\mathcal{G}$), 
since it makes the argument easier (e.g. Remark \ref{rm:dimer-vs-path}), even though the main objects we are studying are $m$-dimer covers.

\begin{theorem}\label{thm 1}
    \[ \llb a_1, \dots, a_n \rrb_m = \# \Omega_m[a_1,\dots,a_n] = \# \mathcal{L}^*_m[a_1,\dots,a_n] \]
\end{theorem}
\begin{proof}
    The fact that $\# \Omega_m = \# \mathcal{L}^*_m$ is just Proposition \ref{prop:dual_snake_bijection}. To see this is equal to $\llb a_1,\dots,a_n \rrb_m$, 
    we will define a map $\varphi \colon \mathcal{L}^*_m[a_1,\dots,a_n] \to \mathcal{L}^*_m[1^n]$, and see that for each $p \in \mathcal{L}^*_m[1^n]$,
    the size of the fiber is precisely $\prod_i \mathrm{wt}_p(e_i)$ (i.e. the corresponding term in $\llb a_1,\dots,a_n\rrb_m$).

    \medskip
    
    First note that if $a_n>1$, then we have the equality of snake graphs $\mathcal{G}[a_1,\dots,a_n] = \mathcal{G}[a_1, a_2, \dots, a_{n-1}, a_n-1, 1]$ 
    since the sign of edge $e_d$ is independent of the snake graph.  Also, if $a_1$ is greater than 1, up to a reflection about the line $y=x$, 
    we get the equality $\mathcal{G}[a_1, a_2, \dots, a_{n-1}, a_n-1, 1] = \mathcal{G}[1, a_1-1, a_2, \dots, a_{n-1}, a_n-1, 1]$ 
    because changing the sign of $e_0$ reverses the role of adding tiles to the N and E from then on, but does not otherwise affect the snake graph.  
    Hence, $\# \Omega_m$ is the same for both $\mathcal{G}[a_1,\dots,a_n]$ and $\mathcal{G}[1, a_1-1, a_2, \dots, a_{n-1}, a_n-1, 1]$.  
    By Lemma \ref{lem:bracket_equivalence}, the quantity $\llb a_1,\dots,a_n \rrb_m$ also does not change
    under this transformation. We may therefore assume that $a_1 = a_n = 1$, which we will do for the remainder of the proof.

    \medskip

    First, we describe the map $\varphi$ in the case $m=1$. We will call a square tile of the snake graph $\mathcal{G}[a_1,\dots,a_n]$ a \emph{corner}
    if the snake changes from going to right to up (or vice versa). In other words, if the previous square is to the left and the following square is above,
    or if the previous square is below and the following square is to the right, then the square is a corner. We also consider the first and last squares corners.
    It follows from Remark~\ref{rem dual G}(b) that the set of corners of the snake graph $\mathcal{G}^*[a_1,\dots,a_n]$ 
    are naturally in bijection with the corners of the zigzag snake graph $\mathcal{G}^*[1^n]$ (here it is important that we assumed $a_1=a_n=1$, since we also count the end tiles as corners).
    For a lattice path $p \in \mathcal{L}^*_1[a_1,\dots,a_n]$, record whether $p$ goes over or under each corner. Then there is a unique lattice path
    in $\mathcal{L}^*_1[1^n]$ which has the same over/under pattern at the corners. Define this unique path to be $\varphi(p)$.
    In general, for $m > 1$, decompose $p$ into $m$ lattice paths, then apply the map above, and take the resulting multiset.

    \medskip

    Now we will count the sizes of the fibers of this map. Fix some $p \in \mathcal{L}^*_m[1^n]$. We wish to count the number of
    $\tilde{p} \in \mathcal{L}^*_m[a_1,\dots,a_n]$ such that $\varphi(\tilde{p}) = p$. Consider an edge labelled $e_i$ in the zigzag snake graph $\mathcal{G}^*[1^n]$.
    We first consider the case that $i \neq 1,n$. Suppose $e_i$ is a vertical edge (the horizontal case is analogous):
    \begin {center}
    \begin {tikzpicture}[scale=0.7]
        \draw (0,0) -- (1,0) -- (1,1) -- (2,1) -- (2,3) -- (1,3) -- (1,2) -- (0,2) -- cycle;
        \draw (0,1) -- (1,1) -- (1,2) -- (2,2);
        \draw (0.7, 1.5) node {$e_i$};
    \end {tikzpicture}
    \end {center}
    If $p$ contains $k$ copies of the edge $e_i$, then any $\tilde{p} \in \varphi^{-1}(p)$ must have $k$ of its $m$ paths go under the corner to the left of $e_i$ 
    and above the corner to the right of $e_i$. The part of $\mathcal{G}^*[a_1,\dots,a_n]$ between these corners (not including the corner tiles) 
    is a horizontal straight segment consisting of $a_i-1$ tiles. By Lemma \ref{lem:straight_snake}, there are $\mathrm{wt}_p(e_i)$ possibilities on this segment.

    \medskip

    If $i=1$ or $i=n$, then $e_i$ is a boundary edge on the first or last square. If $p$ has multiplicity $k$ on edge $e_i$, then this contributes weight $1$
    to the product, since we assumed $a_1=a_n=1$, and $\llb 1 \rrb_m = 1$. On the other hand, any pre-image $\tilde{p}$ of $p$ must have $k$ of its paths
    go under the first square and $m-k$ going over it, so there is only 1 local configuration near the first square.

    \medskip

    Taking the contributions from all segments (for each $a_i$) gives the result.
\end{proof}

\begin{proof}[Proof of Theorem \ref{thm intro1}] 
    We combine together the results of Theorem \ref{thm:matrix} and Theorem \ref{thm 1} to recover our desired equality between the number of $m$-dimer covers 
    of the snake graph $\mathcal G[a_1,a_2,\dots, a_n]$ and the top-left entry of the corresponding matrix product. 
\end{proof}

\section {Lattice of $m$-Dimer Covers and Generating Functions}
\label{sect 5}

There are two natural ways to obtain generating functions from $m$-dimer covers on snake graphs. On one hand, you can
fix a snake graph $\mathcal{G}$, and consider the sequence $\# \Omega_m(\mathcal{G})$ as $m$ varies. On the other hand,
you could fix $m$, and consider $\# \Omega_m(\mathcal{G}_k)$ for some sequence $\mathcal{G}_1, \mathcal{G}_2, \mathcal{G}_3, \dots$ of snake
graphs which follow some natural pattern.  In this section, we consider generating functions using this first approach,  
which is related to the enumeration of certain $P$-partitions
and reverse plane partitions. As we discuss in Section \ref{sec:fixing_m}, the second method has connections to number theory, in particular the ratios of the lengths of 
diagonals in regular polygons.

Recall the definition of $m$-lattice path (Section \ref{sec:mdimer}) on a snake graph. Also recall from Proposition \ref{prop:dual_snake_bijection}
that there is a bijection between $m$-dimer covers of $\mathcal{G}$ and $m$-lattice paths on $\mathcal{G}^*$. We will focus for now on
$m$-lattice paths of the dual $\mathcal{G}^*$.

Every $m$-lattice path $p$ can be seen as a labelling of the tiles of $\calG^*$, as follows. Label each tile of $\mathcal{G}^*$ with the number of
lattice paths in $p$ which go above that tile.

\begin {example} \label{eg:m-path}
    Below are pictures of a 2-lattice path and a {$3$}-lattice path on the same snake graph,
    with the corresponding labelings of the tiles.
    \begin {center}
    \begin {tikzpicture}[scale=0.8]
        \foreach \x/\y in {0/0, 0/1, 0/2, 1/2, 2/2, 2/3, 3/3} {
            \draw (\x,\y) -- (\x+1,\y) -- (\x+1,\y+1) -- (\x,\y+1) -- cycle;
        }

        \draw (0.5,0.5) node {2};
        \draw (0.5,1.5) node {1};
        \draw (0.5,2.5) node {0};
        \draw (1.5,2.5) node {1};
        \draw (2.5,2.5) node {1};
        \draw (2.5,3.5) node {1};
        \draw (3.5,3.5) node {2};

        \draw[blue, line width=1.5] (0,0) -- (0,1);
        \draw[blue, line width=1.5] (-0.1,0) -- (-0.1,1);
        \draw[blue, line width=1.5] (0,1) -- (1,1) -- (1,2);
        \draw[blue, line width=1.5] (0,1) -- (0,2) -- (1,2);
        \draw[blue, line width=1.5] (1,2) -- (3,2) -- (3,4) -- (4,4);
        \draw[blue, line width=1.5] (1,2) -- (1,3) -- (2,3) -- (2,4) -- (3,4);
        \draw[blue, line width=1.5] (3,4.1) -- (4,4.1);

        \begin {scope}[shift={(6,0)}]
            \foreach \x/\y in {0/0, 0/1, 1/1, 1/2, 2/2, 2/3, 3/3} {
                \draw (\x,\y) -- (\x+1,\y) -- (\x+1,\y+1) -- (\x,\y+1) -- cycle;
            }

            \draw[blue, line width=1.5] (0,0) -- (1,0) -- (1,1) -- (2,1) -- (2,4) -- (4,4);
            \draw[blue, line width=1.5] (0,0) -- (0,2) -- (3,2) -- (3,4);
            \draw[blue, line width=1.5] (3,4.1) -- (4,4.1);
            \draw[blue, line width=1.5] (-0.1,0) -- (-0.1,2);
            \draw[blue, line width=1.5] (0,2.1) -- (1.9,2.1);
            \draw[blue, line width=1.5] (1.9,2.1) -- (1.9,3) -- (2.9,3) -- (2.9,4);
            \draw[blue, line width=1.5] (3,4.2) -- (4,4.2);

            \draw (0.5,0.5) node {2};
            \draw (0.5,1.5) node {2};
            \draw (1.5,1.5) node {2};
            \draw (1.5,2.5) node {0};
            \draw (2.5,2.5) node {2};
            \draw (2.5,3.5) node {1};
            \draw (3.5,3.5) node {3};
        \end {scope}
    \end {tikzpicture}
    \end {center}
\end {example}

Note that because the lattice paths must go north-east, the labels in the tiles weakly increase going down the columns and
across the rows (left-to-right). Such labelings are equivalent to the notions of \emph{reverse plane partitions} and \emph{$P$-partitions},
which we will now define.

\begin {definition}
    Let $\lambda$ be a Young diagram (or a skew Young diagram). A \emph{reverse plane partition} of shape $\lambda$
    is a filling of the tiles of $\lambda$ with non-negative integers which are weakly increasing in rows (left-to-right) and columns (top-to-bottom).
\end {definition}

\begin {remark}
    A snake graph is the same thing as a \emph{border strip},
    which is a connected skew Young diagram $\lambda / \mu$ which contains no $2 \times 2$ square.
\end {remark}

\begin {prop} \label{prop:snakes_and_rpps}
    Let $\mathcal{G}$ be a snake graph. There is a bijection between $\mathcal{L}_m(\mathcal{G})$ (the set of $m$-lattice paths on $\mathcal{G}$) 
    and the set of reverse plane partitions of shape $\mathcal{G}$ with all parts less than or equal to $m$.
    Under this bijection the area under the $m$-lattice path is equal to the size of
    the reverse plane partition (the sum of its parts).
\end {prop}

\begin {definition} \cite{stanley_72}
    Given a poset $P$ of size $|P| = p$, a \emph{$P$-partition} is a map $\sigma \colon P \to \Bbb{N}$ which is order-reversing.
    That is, if $x < y$ in $P$, then $\sigma(x) \geq \sigma(y)$.
\end {definition}

\begin {remark}\label{rem 55}
    A snake graph $\mathcal{G}$ (thought of as a skew Young diagram) is naturally thought of as a poset $(P(\mathcal{G}),<)$, where the underlying set $P(\mathcal{G})$ is the set of
    tiles of the diagram. The partial order is given by $x < y$ if tile $x$ lies below or to the right of tile $y$. The posets of the form $P(\mathcal{G})$ are also known as \emph{fence posets}.
    Proposition \ref{prop:snakes_and_rpps} equivalently gives a bijection between $m$-lattice paths on $\mathcal{G}$ and $P$-partitions (with parts at most $m$) on the associated
    fence poset.
\end {remark}

\begin{definition}\label{def:quasi}
Define a family of polynomials in the spirit of Schur polynomials:
\[\mathcal{F}_{\calG^*}(x_0,x_1,\dots,x_m):=\sum_{p\in \mathcal{L}^*_m(\calG)}\prod _{i\in[m]} x_i^{\# \text{ of tiles with label }i\text{ in }p}.\]
\end{definition}
For example, the two lattice paths in Example \ref{eg:m-path} will contribute the monomials $x_0 x_1^4 x_2^2$ and $x_0 x_1 x_2^4 x_3$ respectively.

\begin{remark}
    The polynomials $\mathcal{F}_{\calG^*}(x_0,x_1,\dots,x_m)$ are quasi-symmetric, 
    as they are special cases of Lam-Pylyavskyy's \cite{lp_08} \emph{wave Schur polynomials} of ribbon shape.  
    By extending our sum to run over $m$-lattice paths of all $m\in\mathbb{N}$, we are able to apply Theorem 6.6 of \cite{lp_08} 
    (which gives Jacobi-Trudi formulae for general wave Schur functions) thus providing a determinantal formula for enumerating $m$-dimer covers or $m$-lattice paths. 
\end{remark}

We now review some classical results in the theory of $P$-partitions, so that we may ultimately connect them back to 
the enumeration of $m$-dimer covers.

In \cite{stanley_72}, Stanley gave formulas for the generating functions associated to $P$-partitions. 
Let $U_{P;m}(q)$ denote the generating function $\sum q^{|\sigma|}$ where $|\sigma|$ is the sum of the parts/labels,
and the sum is over all $P$-partitions with all parts less than or equal to $m$. 

In order to state Stanley's results, we need a little notation.

Let $L(P)$ be the set of linear extensions of $P$. That is, order-preserving bijections $P \to [p]=\{1,2,\ldots,p\}$. Fixing some
linear extension $\pi$, any other linear extension $\sigma$ is associated to a permutation $\pi \circ \sigma^{-1} \in S_p$.
The \emph{major index} of a permutation, denoted $\mathrm{maj}(\sigma)$, is the sum of the positions of the descents.
Let $W_i(q)$ be the major-index generating function 
\[ W_i(q) = \sum q^{\mathrm{maj}(\sigma)} \]
where $\mathrm{maj}(\sigma)$ really means the major index of the permutation $\pi \circ \sigma^{-1}$, and the sum is over
all linear extensions with exactly $i$ descents. Stanley gave the following formula for the generating function $U_{P;m}(q)$.

\begin {prop} [\cite{stanley_72}, Proposition 8.2] \label{prop:stanley_Um}
    Let $P$ be a poset of size $p$. The generating function $U_{P;m}(q)$ is given by
    \[ U_{P;m}(q) = \sum_{i=0}^{p-1} \binom{p+m-i}{p}_q W_i(q) \]
    where $\binom{n}{k}_q$ are the Gaussian $q$-binomial coefficients.
\end {prop}

Stanley also gave a formula for the multivariate generating function $\sum_m U_{P;m}(q) x^m$.

\begin {theorem} [\cite{stanley_72}, Proposition 8.3] \label{thm:stanley_gen_fn}
    Let $P$ be a poset of size $p$. Then
    \[ \sum_{m=0}^\infty U_{P;m}(q) x^m = \frac{1}{(x;q)_{p+1}} \sum_{\sigma \in L(P)} q^{\mathrm{maj}(\sigma)} x^{\mathrm{des}(\sigma)} \]
    where $(x;q)_k$ is the $q$-Pochhammer symbol
    \[ (x;q)_k := (1-x)(1-qx) \cdots (1-q^{k-1}x). \]
\end {theorem}

Finally, we discuss the application of this formula to the problem of counting $m$-dimer covers on a snake graph $\mathcal{G}$, or equivalently
$m$-lattice paths on $\mathcal{G}^*$. Recall that $P(\mathcal{G})$ is the poset of the snake graph $\mathcal{G}$ introduced in Remark~\ref{rem 55}.

\begin{remark}
The poset of $\Omega_m(\calG)$ or $\mathcal{L}_m(\calG^*)$ is always a distributive lattice, see \cite{propp02}. For $m=1$, it is well known that this lattice is isomorphic to the order ideals of  the poset $P(\calG^*)$ (see \cite{claussen20}). In general, the poset of $\Omega_m(\calG)$ is isomorphic to the lattice of order ideals of the Cartesian product $P(\calG^*)\times [m]$, where $[m]$ is the $m$-element chain. The case of $m=2$ of this latter result is proved in Section 9 of \cite{moz2} and the statement can be easily generalized to arbitrary $m$.
\end{remark}

By Proposition \ref{prop:snakes_and_rpps}, there is a weight-preserving bijection between the set of $P$-partitions on $\mathcal{G}^*$ whose parts are at most $m$
and $\mathcal{L}_m(\mathcal{G}^*)$. Therefore $U_{P(\mathcal{G^*});m}(q)$ is the rank generating function for the poset structure on $\mathcal{L}_m(\mathcal{G^*})$,
or equivalently on $\Omega_m(\mathcal{G})$.

Applying Theorem \ref{thm:stanley_gen_fn} in this special case gives the following.

\begin {corollary} \label{cor:dimer_q_gen_fn}
    Let $\mathcal{G} = \mathcal{G}[a_1,\dots,a_n]$, and let $N := \sum_{i=1}^n a_i$. 
    The generating function for $m$-lattice paths on $\mathcal{G}^*$, or $m$-dimer covers on $\mathcal{G}$, is given by
    \[ F_{\calG^*}(q,x) = \sum_{m=0}^\infty U_{P(\calG^*);m}(q)x^m=\frac{1}{(x;q)_N} \sum_{\sigma \in L(P(\mathcal{G}^*))} q^{\mathrm{maj}(\sigma)} x^{\mathrm{des}(\sigma)} \]
\end {corollary}

In particular, when $q=1$, we get the generating function for {the number of} $m$-dimer covers on $\mathcal{G}$ (or $m$-lattice paths on $\mathcal{G}^*$):

\begin {corollary} \label{cor:dimer_gen_fn}
    Let $\mathcal{G} = \mathcal{G}[a_1,\dots,a_n]$, and let $N := \sum_{i=1}^n a_i$. Then
    \[F(x)= F_{\calG^*}(1,x) = \frac{1}{(1-x)^N} \sum_{\sigma \in L(P(\mathcal{G}^*))} x^{\mathrm{des}(\sigma)} \]
\end {corollary}

\begin{remark}
	{The rank generating function $U_{P(\calG^*);m}(q)$ is} also the principal specialization of the quasi-symmetric polynomial defined in 
	Definition \ref{def:quasi}, that is $U_{P(\calG^*);m}(q) = \mathcal{F}_{\calG^*}(1,q,q^2,\dots,q^m)$. 
	{Furthermore, $\lim_{m \to \infty} U_{P(\calG^*);m}(q)$ yields a power series (which is a quasi-symmetric function) that is the principal specialization $\mathcal{F}_{\calG^*}(1,q,q^2,\dots)$, and there is a combinatorial formula given by Morales-Pak-Panova \cite[Theorem 1.5] {mpp_18}.}
\end{remark}

\begin {example}
    Consider the snake graph which is a vertical column of $n$ tiles. The corresponding poset is the chain with vertices labeled $1,2,\dots,n$ from bottom-to-top.
    There is only one linear extension, given by the identity permutation $123 \cdots n$ (which has no descents). In this case $N = n+1$, and so the right-hand side of
    Corollary \ref{cor:dimer_gen_fn} is $\frac{1}{(1-x)^{n+1}}$. By the binomial theorem, this is equal to
    \[ \frac{1}{(1-x)^{n+1}} = \sum_{m \geq 0} \binom{n+m}{m} x^m \]
    This agrees with Lemma \ref{lem:straight_snake}, which says that there are $\binom{n+m}{m}$ $m$-lattice paths on this snake graph. 
    Equivalently, there are $\binom{n+m}{m}$ $m$-dimer covers
    on the zigzag snake graph with $n$ tiles. The $q$-version (Corollary \ref{cor:dimer_q_gen_fn}) says that the generating function is $\frac{1}{(x;q)_{n+1}}$, 
    which by the $q$-binomial theorem, is
    \[ \frac{1}{(x;q)_{n+1}} = \frac{1}{(1-x)(1-qx)(1-q^2x) \cdots (1-q^nx)} = \sum_{m \geq 0} \binom{n+m}{m}_q x^m \]
    This means the rank generating function for $m$-dimer covers on the zigzag $\mathcal{G}$ is the $q$-binomial coefficient
    \[ \sum_{M \in \Omega_m(\mathcal{G})} q^{|M|} = \binom{n+m}{m}_q \]
    This also just follows directly from Proposition \ref{prop:stanley_Um} and the fact that $W_0(q) =1$ and $W_i(q) = 0$  when $i \geq 1$ in this case.
    For example, when $n=3$ and $m=2$, we have $\binom{5}{2}_q = 1 + q + 2q + 2q^2 + 2q^3 + q^4 + q^5$. This is the rank function of the poset of double-dimer covers 
    on the zigzag snake graph with 3 tiles, as in Figure \ref{fig:double_dimer_poset}.
    \begin {figure}[h!]
    \centering
    \begin{tikzpicture}[scale=0.45]
    
    \foreach \x/\y in {-3/0,0/3,3/6,-3/6,0/9,6/9,3/12,-3/12,0/15,-3/18} {
    \draw [line width = 0.4] (1+\x,0+\y) --(1+\x,1+\y) -- (\x+0,1+\y) -- (0+\x,0+\y) -- (2+\x,0+\y) -- (2+\x,2+\y) -- (1+\x,2+\y) -- (1+\x,1+\y) -- (2+\x,1+\y);
    }
     
    \begin{scope}[shift={(-3,0)}]
    \draw [red,double,double distance = 0.15 em, line width = 2] (0,0) -- (1,0);
    \draw [red,double,double distance = 0.15 em, line width = 2] (0,1) -- (1,1);
    \draw [red,double,double distance = 0.15 em, line width = 2] (2,0) -- (2,1);
    \draw [red,double,double distance = 0.15 em, line width = 2] (1,2) -- (2,2);
    \end{scope}
    
     \begin{scope}[shift={(-3,6)}]
    \draw [red,double,double distance = 0.15 em, line width = 2] (0,0) -- (0,1);
    \draw [red,double,double distance = 0.15 em, line width = 2] (1,0) -- (1,1);
    \draw [red,double,double distance = 0.15 em, line width = 2] (2,0) -- (2,1);
    \draw [red,double,double distance = 0.15 em, line width = 2] (1,2) -- (2,2);
    \end{scope}
    
    \begin{scope}[shift={(-3,12)}]
    \draw [red,double,double distance = 0.15 em, line width = 2] (0,0) -- (0,1);
    \draw [red,double,double distance = 0.15 em, line width = 2] (1,0) -- (2,0);
    \draw [red,double,double distance = 0.15 em, line width = 2] (2,1) -- (1,1);
    \draw [red,double,double distance = 0.15 em, line width = 2] (1,2) -- (2,2);
    \end{scope}
    
    \begin{scope}[shift={(-3,18)}]
    \draw [red,double,double distance = 0.15 em, line width = 2] (0,0) -- (0,1);
    \draw [red,double,double distance = 0.15 em, line width = 2] (1,0) -- (2,0);
    \draw [red,double,double distance = 0.15 em, line width = 2] (2,1) -- (2,2);
    \draw [red,double,double distance = 0.15 em, line width = 2] (1,2) -- (1,1);
    \end{scope}
    
    \begin{scope}[shift={(0,3)}]

    \draw [red, double,double distance = 0.15 em, line width = 2] (2,0) -- (2,1);
    \draw [red,double,double distance = 0.15 em, line width = 2] (1,2) -- (2,2);
    \draw [red,line width = 2] (0,0) -- (0,1) -- (1,1) -- (1,0) -- cycle;
    \end{scope}
    \begin{scope}[shift={(0,9)}]
    \draw [red,double,double distance = 0.15 em, line width = 2] (0,0) -- (0,1);

    \draw [red,double,double distance = 0.15 em, line width = 2] (1,2) -- (2,2);
    \draw [red,line width = 2] (1,0) -- (2,0) -- (2,1) -- (1,1) -- cycle;
    \end{scope}
    
    \begin{scope}[shift={(0,15)}]
    \draw [red,double,double distance = 0.15 em, line width = 2] (0,0) -- (0,1);
    \draw [red,double,double distance = 0.15 em, line width = 2] (1,0) -- (2,0);

    \draw [red,line width = 2] (1,1) -- (2,1) -- (2,2) -- (1,2) -- cycle;
    \end{scope}
     \begin{scope}[shift={(3,6)}]

    \draw [red,double,double distance = 0.15 em, line width = 2] (1,2) -- (2,2);
    \draw [red,line width = 2] (0,0) -- (0,1) -- (2,1) -- (2,0) -- cycle;
    \end{scope}
    
    \begin{scope}[shift={(3,12)}]
    \draw [red,double,double distance = 0.15 em, line width = 2] (0,0) -- (0,1);
    \draw [red,line width = 2] (1,0) -- (2,0) -- (2,2) -- (1,2) -- cycle;
    \end{scope}
    
    \begin{scope}[shift={(6,9)}]
    \draw [red,line width = 2] (0,0) -- (2,0) -- (2,2) -- (1,2) -- (1,1) -- (0,1) -- cycle;
    \end{scope}
      
     \draw (-0.5,1.5) -- (0.5,2.5);
     \draw (-0.5,5.5) -- (0.5,4.5);
     \foreach \j/\i in {0/6,0/12,3/9,3/3,6/6}{\draw (-0.5+\j,1.5+\i) -- (0.5+\j,2.5+\i);
     \draw (-0.5+\j,5.5+\i) -- (0.5+\j,4.5+\i);}
    \end{tikzpicture}    
    \caption {Poset structure of $\Omega_2[4]$}
    \label {fig:double_dimer_poset}
    \end {figure}
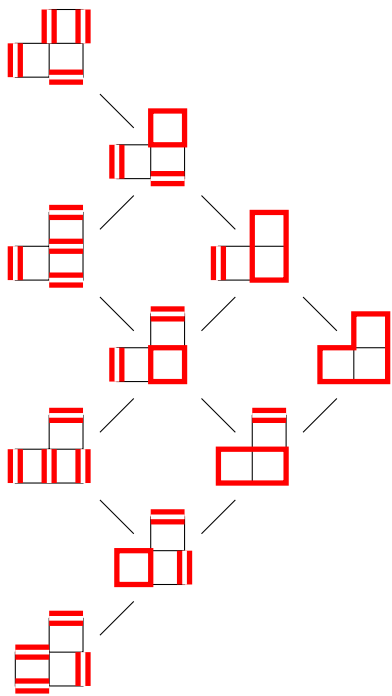
\end {example}

\begin {example}
    We could consider the other extreme, where we switch the roles of $\mathcal{G}$ and $\mathcal{G}^*$ in the
    previous example, so that $\mathcal{G}^*$ is a zigzag snake graph, and $\mathcal{G}$ is a straight snake graph. In this case,
    Corollary \ref{cor:dimer_gen_fn} says that $F(x) = \frac{f(x)}{(1-x)^{n+1}}$, where $f(x)$ is the descent generating function
    for linear extensions of the poset $\mathcal{P}(\mathcal{G}^*)$. This is the poset with elements $b_1,\dots,b_n$ and relations
    $b_1 < b_2 > b_3 < b_4 > b_5 < \cdots$. For example, when $n=4$, if we take $b_1 < b_3 < b_2 < b_4$ as our distinguished
    linear extension $\pi$, then there are 5 linear extensions of the poset, illustrated in Figure \ref{fig:lin_ext}.

    \begin {figure}
    \centering
    \begin {tikzpicture}
        \begin {scope}[shift={(9,1)}]
            \draw (0,0) node{$
            \begin {array} {ccc}
                \sigma & \mathrm{maj}(\sigma) & \mathrm{des}(\sigma) \\ \hline
                1234   & 0                    & 0 \\
                1243   & 3                    & 1 \\
                2134   & 1                    & 1 \\
                2413   & 2                    & 1 \\ 
                2143   & 4                    & 2 \\
            \end {array}
            $};
        \end {scope}

        \draw (0,0) node {$b_1$};
        \draw (1.5,1.5) node {$b_2$};
        \draw (3,0) node {$b_3$};
        \draw (4.5,1.5) node {$b_4$};

        \draw (0.2,0.2) -- (1.3,1.3);
        \draw (1.7,1.3) -- (2.8,0.2);
        \draw (3.2,0.2) -- (4.3,1.3);

        \draw (0,-0.5) node {1};
        \draw (1.5,2) node {3};
        \draw (3,-0.5) node {2};
        \draw (4.5,2) node {4};
    \end {tikzpicture}
    \caption {The linear extensions of a fence poset. The labels next to the $b_i$'s are the values of the distinguished linear extension $\pi$.}
    \label {fig:lin_ext}
    \end {figure}

    Therefore $f(x) = 1+3x+x^2$, and we have
    \[ F(x) = \frac{1 + 3x + x^2}{(1-x)^5} = 1 + 8x + 31x^2 + 85x^3 + 190x^4 + \cdots \]
    This means the snake graph which is a horizontal row of 4 tiles has $8$ dimer covers, $31$ double-dimer covers, $85$ triple-dimer covers, etc.
    If we include the $q$-variable and take the maj statistic into account, we get
    \begin {align*} 
        F(q,x) &= \frac{1 + q[3]_q x + q^4 x^2}{(x;q)_5} \\
               &= 1 + (1+2q+2q^2+2q^3+q^4)x \\
               &\phantom{=} + (1+2q+4q^2+5q^3+7q^4+5q^5+4q^6+2q^7+q^8)x^2 + \cdots 
    \end {align*}
    This means for example that the 8 dimer covers of $\mathcal{G}$ have a poset structure with rank function $1+2q+2q^2+2q^3+q^4$,
    and the 31 double-dimer covers have poset structure with rank function $1+2q+4q^2+5q^3+7q^4+5q^5+4q^6+2q^7+q^8$.
\end {example}

\begin {remark}
    In \cite{moz2}, it was shown that the $\lambda$-lengths in the decorated super Teich\-m\"{u}ller space satisfy a Laurent phenomenon,
    and that the terms in the Laurent expressions are given as the weights of double dimer covers of snake graphs. Analogous to the case
    with ordinary cluster algebras, the terms in these Laurent expressions have a natural partial order. These posets are precisely 
    $\Omega_2(\mathcal{G})$ {and $\mathcal{L}_2(\mathcal{G}^*)$ discussed above}. In particular, the $x^2$-coefficient of the generating
    function $F(q,x)$ is a version of the $F$-polynomial for these super cluster variables.
\end {remark}

\section{Continued Fractions, Fibonacci Numbers and Golden Ratio}
\label{sec:fixing_m}

We now consider generating functions where we fix $m$, and look at the sizes of $\Omega_m(\mathcal{G}_k)$ for some natural sequence of snake graphs
$\mathcal{G}_1, \mathcal{G}_2, \mathcal{G}_3, \dots$. As in the previous section, we consider the two natural examples where either
$\mathcal{G}_n$ or $\mathcal{G}_n^*$ is a straight snake graph with $n$ tiles (and the dual is a zigzag).

\begin {example} \label{ex:fibonacci}
    When $\mathcal{G}_k$ is the straight snake graph with $k$ tiles (and $\mathcal{G}^*_k$ is a zigzag), and $m=1$, it is well-known 
    that $\left| \Omega_1(\mathcal{G}_k) \right|$ is the Fibonacci sequence. If we $q$-count the dimer covers (i.e. consider the rank generating
    function of $\Omega_1(\mathcal{G}_k)$), then we get the sequence of polynomials 
    \[ f_k(q) = \sum_{d \in \Omega_1(\mathcal{G}_k)} q^{\mathrm{rk}(d)} \]
    These polynomials are the numerators of the $q$-deformed ratios of Fibonacci numbers $\left[ \frac{f_k}{f_{k-1}} \right]_q$
    in the language of Morier-Genoud and Ovsienko's $q$-rationals (see \cite{mgo_20}, section 6). The denominators are the rank generating
    functions of the dual posets (equivalently they are obtained from these polynomials by reversing the order of the coefficients).
    The first several polynomials are
    \begin {align*}
        f_1(q) &= 1 + q \\
        f_2(q) &= 1 + q + q^2 \\
        f_3(q) &= 1 + q + 2q^2 + q^3 \\
        f_4(q) &= 1 + 2q + 2q^2 + 2q^3 + q^4 \\
        f_5(q) &= 1 + 2q + 3q^2 + 3q^3 + 3q^4 + q^5 \\
        f_6(q) &= 1 + 3q + 4q^2 + 5q^3 + 4q^4 + 3q^5 + q^6
    \end {align*}
    They satisfy the following $q$-deformed version of the Fibonacci recurrence:
    \[
        f_n(q) = \begin{cases}
            f_{n-1}(q) + q^2 f_{n-2}(q) & \text{ if $n$ is odd} \\[1ex]
            q f_{n-1}(q) + f_{n-2}(q) & \text{ if $n$ is even}
        \end{cases}
    \]
\end {example}

\begin {example} \label{ex:straight_double_dimers}
    Consider again the sequence $\mathcal{G}_k$ of straight snake graphs, as in the previous example, but now fix $m=2$.
    That is, we are considering the double-dimer covers of straight snake graphs. By Theorem \ref{thm:matrix}, the top-left 
    entry of the matrix $\Lambda^{(2)}(1)^{k+1}$ is equal to $\left| \Omega_2(\mathcal{G}_k) \right|$, the number of double-dimer
    covers of the straight snake graph with $k$ tiles. 
    It is a simple calculation to see that
    \[ 
        (I - x \Lambda^{(2)}(1))^{-1} = \frac{1}{1-2x-x^2+x^3}
        \begin{pmatrix} 
            1-x & x & x-x^2 \\[1ex]
            x & 1-x-x^2 & x^2 \\[1ex] 
            x-x^2 & x^2 & 1-2x 
        \end{pmatrix}
    \]
    Using the standard matrix identity $\sum_{k \geq 0} M^{k+1} x^k = M(I - xM)^{-1}$ when $M$ is square, the equation above implies that 
    the sum $\sum_{k \geq 0} \Lambda^{(2)}(1)^{k+1} x^k$ is given by the following matrix of rational functions:
    \[ 
        \sum_{k \geq 0} \Lambda^{(2)}(1)^{k+1}x^k = \frac{1}{1-2x-x^2+x^3} 
        \begin{pmatrix} 
            1+x-x^2 & 1 & 1-x \\[1ex]
            1 & 1-x^2 & x \\[1ex] 
            1-x & x & x-x^2 
        \end{pmatrix} 
    \]
    In particular, looking at the top-left 
    entry, we see that the generating function for the sequence $\left| \Omega_2(\mathcal{G}_k) \right|$ is
    given by
    \begin{equation} \label{Eq:13614}
     \sum_{k \geq 0} \left| \Omega_2(\mathcal{G}_k) \right| x^k = \frac{1+x-x^2}{1-2x-x^2+x^3} = 1 + 3x + 6x^2 + 14x^3 + 31x^4 + \cdots 
     \end{equation}
\end {example}

See OEIS A006356 or \cite{BK76} where an interpretation in terms of order ideals in a distributive lattice appears for the sequence associated to this generating function.
We return to our original question concerning sizes of $\Omega_m(\mathcal{G}_k)$ for sequences of snake graphs $\mathcal{G}_1, \mathcal{G}_2, \mathcal{G}_3, \dots$ 
after introducing some notation motivated by the above.

\subsection{Higher Continued Fractions from Snake Graphs}

As described just above Equation (\ref{eq:contfrac}), an ordinary continued fraction $[a_1,a_2,\dots,a_n]$ can be expressed as $\frac{p_n}{q_n}$, 
and we can think of this projectively as either the element $(p_n : q_n)$ or $(\frac{p_n}{q_n}: 1)$ in $\Bbb{QP}^1$. This is the projectivization
of the first column of the $2 \times 2$ matrix in Equation (\ref{eq mat}).
We now extend this to higher $m$, considering the projectivization of the first column of the $(m+1) \times (m+1)$ matrix $\Lambda^{(m)}(a_1) \cdots \Lambda^{(m)}(a_n)$
as a generalized continued fraction.

\begin {definition} \label{def:r_m}
        For a list of positive integers $a_1,a_2,\dots,a_n$, 
    the corresponding \emph{finite $m$-dimensional continued fraction} is a vector\footnote{One can also view $\mathrm{CF}_m$ as the homogenous coordinates of a point in projective space.} in $\Bbb{Q}^{m+1}$,
    \[ \mathrm{CF}_m(a_1,\dots,a_n) := \Big( r_{m,m}(a_1,\dots,a_n), \dots, r_{2,m}(a_1, \dots, a_n), r_{1,m}(a_1,\dots,a_n), r_{0,m}(a_1,\dots,a_n)\Big) \]
    whose entries are defined recursively as follows.
    For base cases, we set $r_{0,m}(a_1,\dots,a_n) = 1$ (for all $n$), and $r_{k,m}(a) = \llb a \rrb_k = \mbinom{a}{k}$. For $k > 0$ or $n > 1$, we define
    \[ r_{k,m}(a_1,\dots,a_n) = \sum_{i=0}^k \llb a_1\rrb_i \cdot \frac{r_{m-k+i,m}(a_2,\dots,a_n)}{r_{m,m}(a_2,\dots,a_n)} \]
\end {definition}

For the sake of brevity, we will write $r_m := r_{m,m}$. 

\begin{theorem} \label{thm:r_m}
    For $0 \leq i \leq m$, the values recursively defined in Definition \ref{def:r_m} are ratios of matrix elements, i.e. 
    \[r_{i,m}(a_1,\dots,a_n)={X_{m+1-i,1}\over X_{m+1,1}}\]
    where 
    $X=\Lambda^{(m)}(a_1)\cdots\Lambda^{(m)}(a_n)$ is the matrix from Theorem \ref{thm:matrix}.
    In the special case that $i=m$, we obtain
    \[r_m(a_1,\dots, a_n)={\# \Omega_m(\mathcal{G}[a_1,\dots,a_n])\over \# \Omega_m(\mathcal{G}[a_2,\dots,a_n])}.\]
\end{theorem}

\begin{proof}
    We begin with the first statement.
    Induct on $n$. When $n=1$, we have by definition that $r_{k,m}(a_1) = \llb a_1 \rrb_k$, which are the entries in the first
    column of $\Lambda^{(m)}(a_1)$.

    Now suppose the result is true for $Y = \Lambda^{(m)}(a_2) \cdots \Lambda^{(m)}(a_n)$; that is, $r_{i,m}(a_2,\dots,a_n) = \frac{Y_{m+1-i,1}}{Y_{m+1,1}}$. 
    The entries in the first column of the product $X = \Lambda^{(m)}(a_1) Y$ are then
    \begin {align*}
        X_{m+1-i,1} &= \sum_{k=1}^{m+1} \Lambda^{(m)}(a_1)_{m+1-i,k} Y_{k,1} \\
                    &= \sum_{k=1}^{i+1} \llb a_1 \rrb_{i-k+1} Y_{k,1} \\
                    &= \sum_{j=0}^i \llb a_1 \rrb_j Y_{i-j+1,1} \\
                    &= \sum_{j=0}^i \llb a_1 \rrb_j r_{m+j-i,m}(a_2,\dots,a_n) Y_{m+1,1} \tag{induction}
    \end {align*}
    Now we divide by $X_{m+1,1}$, and note that $X_{m+1,1} = Y_{1,1}$, to get
    \[ \frac{X_{m+1-i,1}}{X_{m+1,1}} = \sum_{j=0}^i \llb a_1 \rrb_j r_{m+j-i,m}(a_2,\dots,a_n) \frac{Y_{m+1,1}}{Y_{1,1}} \]
    Lastly, we use induction again to note that $\frac{Y_{m+1,1}}{Y_{1,1}} = \frac{1}{r_m(a_2,\dots,a_n)}$.
    This proves that  
    \[\frac{X_{m+1-i,1}}{X_{m+1,1}} = \sum_{j=0}^i \llb a_1 \rrb_j \frac{r_{m+j-i,m}(a_2,\dots,a_n)}{r_{m}(a_2,\dots,a_n)} =  r_{i,m}(a_1,\dots, a_n).\]
    We also note that in the special case that $i=m$, we get $r_{m}(a_1,\dots, a_n) = \frac{X_{1,1}}{X_{m+1,1}} = \frac{X_{1,1}}{Y_{1,1}}$, and thus 
    the second statement follows by two applications of Theorem \ref{thm intro1}.  
\end{proof}

\begin {remark}
    For each $m$, we can define a map $\Bbb{Q}_{\geq 1} \to \Bbb{Q}_{\geq 1}$, which we also denote $r_m$, as follows.
    If $\alpha \in \Bbb{Q}_{\geq 1}$ has continued fraction $\alpha = [a_1,a_2,\dots,a_n]$, then we define $r_m(\alpha) := r_m(a_1,\dots,a_n)$.
    By construction, if $X = \Lambda^{(m)}(a_1) \cdots \Lambda^{(m)}(a_n)$, then $r_m(\alpha) = \frac{X_{11}}{X_{m+1,1}}$.
    Note that Lemma \ref{lem:bracket_equivalence} and Lemma \ref{lemma:new} (see below) show that this is well-defined, since the two different continued fraction representations
    of $\alpha \in \Bbb{Q}_{\geq 1}$ will give the same result and the image will be greater than one.
\end {remark}

\begin{remark}
    Furthermore, using Remark \ref{rem:allri}, the values of $r_{i,m}(\alpha)  = \frac{X_{i1}}{X_{m+1,1}}$ are similarly well-defined, and do not depend on the specific choice of 
    continued fraction representation of $\alpha \in \Bbb{Q}_{\geq 1}$.
\end{remark}

Now let's look at the recurrence for some small values of $m$. For visual reference, the graphs of the maps $r_2(x)$, $r_3(x)$ and $r_4(x)$ are pictured 
in Figure \ref{fig:r2_graph}.

\begin {figure}
\centering
    \includegraphics[scale=0.3]{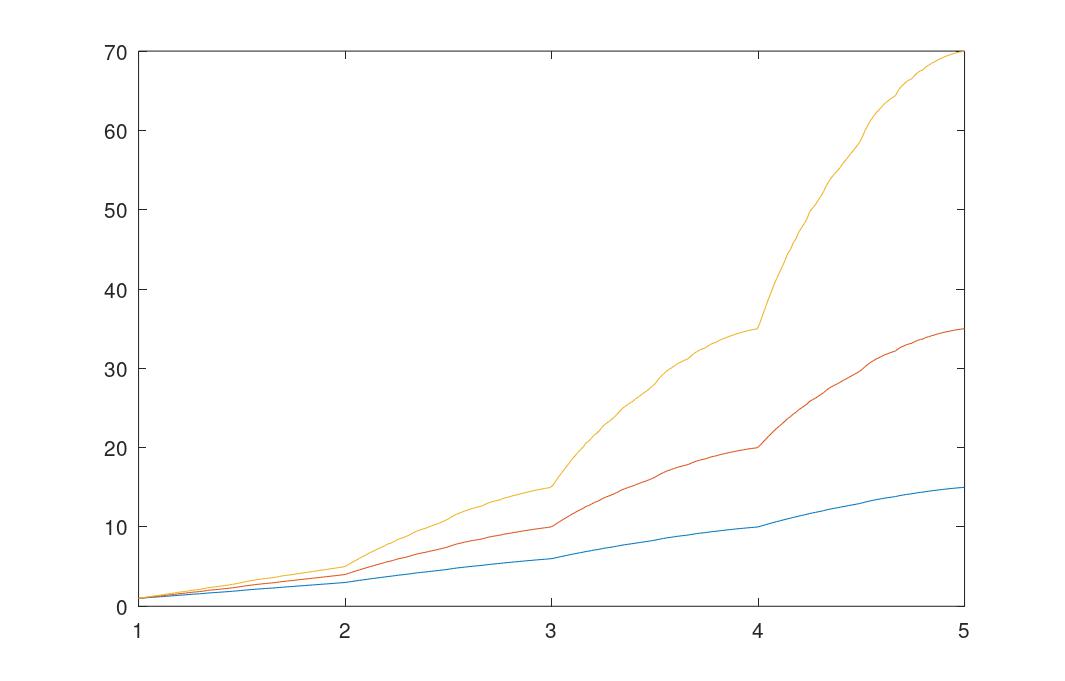}
\caption {The graphs of $x \mapsto r_m(x)$ for $m=2,3,4$ and $x \in [1,5]$. The graph of $r_2$ is in blue (bottom),
$r_3$ is orange (middle), and $r_4$ is yellow (top).}
\label {fig:r2_graph}
\end {figure}

\begin{example}[$m=1$]
        The map $r_1 \colon \Bbb{Q}_{\geq 1} \to \Bbb{Q}_{\geq 1}$ is simply the identity function. In other words, $r_1(a_1,\dots,a_n) = [a_1,\dots,a_n]$. Indeed, we have
	\[r_1(a_1,\dots,a_n)={\llb a_1\rrb_1\cdot r_1(a_2,\dots,a_n)+1\over r_1(a_2,\dots,a_n)}=a_1+{1\over r_1(a_2,\dots,a_n)}, \]
        which is the usual recursive expression for a continued fraction.
\end{example}

\begin{example}[$m=2$]
	$\mathrm{CF}_2(a_1,\dots,a_n) = \left( r_2(a_1,\dots,a_n), \, r_{1,2}(a_1,\dots,a_n), \, 1 \right)$, where $r_{1,2}$ and $r_2$ satisfy the recurrences
	\[
            r_{1,2}(a_1,\dots,a_n)={\llb a_1\rrb_1 \cdot r_2(a_2,\dots,a_n) + r_{1,2}(a_2,\dots,a_n) \over r_2(a_2,\dots,a_n)}
                = a_1 + \frac{r_{1,2}(a_2,\dots,a_n)}{r_2(a_2,\dots,a_n)}
        \]
	\begin {align*}
            r_2(a_1,\dots,a_n) &= {\llb a_1\rrb_2 \cdot r_2(a_2,\dots,a_n)+\llb a_1 \rrb_1\cdot r_{1,2}(a_2,\dots,a_n)+ 1\over r_2(a_2,\dots,a_n)} \\
                &=  \mbinom{a_1}{2} + a_1 \frac{r_{1,2}(a_2,\dots,a_n)}{r_2(a_2,\dots,a_n)} + \frac{1}{r_2(a_2,\dots,a_n)} 
        \end {align*}
\end{example}

\begin{example}[$m=3$] Let $\mathrm{CF}_3(a_1,\dots,a_n) = (r_3,r_{2,3},r_{1,3},1)$ and $\mathrm{CF}_3(a_2,\dots,a_n) = (r_3',r_{2,3}',r_{1,3}',1)$.  Then we can express the relevant values as

\[r_{1,3} = a_1 + \frac{r_{2,3}'}{r_3'}  \]

\[ r_{2,3} = \mbinom{a_1}{2} + a_1 \frac{r_{2,3}'}{r_3'} + \frac{r_{1,3}'}{r_3'} \]

\[r_3 = \mbinom{a_1}{3} + \mbinom{a_1}{2} \frac{r_{2,3}'}{r_3'} + a_1 \frac{r_{1,3}'}{r_3'} + \frac{1}{r_3'}  \]
\end{example}

\begin{remark}
There are numerous places in the literature where higher dimensional continued fractions, or ternary continued fractions (in the special case of $m=2$), have previously been defined.  
In 1868, Jacobi presented an algorithm using linear recurrences for periodic infinite ternary continued fractions \cite{jacobi1868allgemeine}.  
This algorithm in modern language uses matrix multiplication similar to our approach, except instead of the fundamental matrix being $\Lambda^{(2)}(a)$, it is the $3$-by-$3$ matrix
$\left[ \begin{matrix} a_0 & 1 & 0 \\ b_0 & 0 & 1 \\ 1 & 0 & 0 \end{matrix}\right]$, as explained in \cite[Sec. 5.2]{karpenkov2021hermite}.  
See also \cite[Sec. 2.1]{karpenkov2021hermite} for a related formulation of Jacobi's algorithm using fractional parts analogous to the Euclidean algorithm, 
or \cite{klein1896representation} for an approach to a geometric formulation of multidimensional continued fractions due to Klein, and more recent work \cite{simplest-2-dim} 
showcasing a nice family of such examples in the $m=2$ case.

We note that our construction is in some sense an example of a \emph{multidimensional continued fraction expansion} as defined in \cite{brentjes1981multi}.
It is defined as a sequence of bases of the lattice $\Bbb{Z}^n$ such that a given ray is in the positive span of each basis in the sequence, and that each basis
is obtained from the previous by some elementary row operations. In this sense, the columns of our matrices $X_n = \Lambda^{(m)}(a_1)\Lambda^{(m)}(a_2) \cdots \Lambda^{(m)}(a_n)$
form the sequence of bases, and the ray in the direction of the vector $(r_m(x), r_{m-1,m}(x), \dots, r_{2,m}(x), r_{1,m}(x), 1)$ is at every step in the positive span
of the columns.
\end{remark}

\subsection{Limiting Values}

As in Theorem \ref{thm:r_m}, the values $r_{i,m} \in \Bbb{Q}_{\geq 1}$ can be given as the ratio of matrix entries 
$\frac{X_{m+1-i,1}}{X_{m+1,1}}$, i.e. the ratio of entries in the first column of the matrix $X = \Lambda^{(m)}(a_1)\cdots \Lambda^{(m)}(a_n)$, for 
a given $[a_1,a_2,\dots, a_n]$.  We now show that these definitions can be extended to infinite continued fractions as well.  

\begin {lemma} \label{lemma:new}
    Let $a_1,a_2,\dots,a_n$ be any finite sequence of positive integers. Then the entries of the matrix $X = \Lambda^{(m)}(a_1) \cdots \Lambda^{(m)}(a_n)$ 
    decrease along rows and columns. That is, if $i < i'$ and $j < j'$, then $X_{ij} \geq X_{i'j}$ and $X_{ij} \geq X_{ij'}$.
\end {lemma}
\begin {proof}
    We will induct on $n$. For $n=1$ it is evident from the form of the matrix $\Lambda^{(m)}(a)$, and the fact that
    the sequence $\mbinom{a}{k}$ is increasing.

    Now suppose the entries of $Y = \Lambda^{(m)}(a_1) \cdots \Lambda^{(m)}(a_{n-1})$ decrease along rows and columns, and let $X = Y \Lambda^{(m)}(a_n)$. We have
    \begin {align*} 
        X_{ij} &= \sum_k Y_{ik} \Lambda^{(m)}(a_n)_{kj} \\
               &= \sum_k Y_{ik} \mbinom{a_n}{m+2-k-j} \\
               &> \sum_k Y_{ik} \mbinom{a_n}{m+2-k-j'} \\
               &= X_{ij'} 
    \end {align*}
    The inequality in the middle step follows from the fact that $\mbinom{a_n}{k}$ increases as a function of $k$, and the fact that the entries of $Y$ are positive
    (we did not in fact require the assumption that the entries of $Y$ decrease along rows or columns).
    This shows that the entries of $X$ decrease along rows. Similarly, we have
    \begin {align*} 
        X_{ij} &= \sum_k Y_{ik} \Lambda^{(m)}(a_n)_{kj} \\
               &= \sum_k Y_{ik} \mbinom{a_n}{m+2-k-j} \\
               &\geq \sum_k Y_{i'k} \mbinom{a_n}{m+2-k-j} \\
               &= X_{i'j} 
    \end {align*}
    This time the inequality follows from the assumption that $Y_{ik} \geq Y_{i'k}$ if $i < i'$. This shows the entries of $X$ decrease along columns.
\end {proof}

\begin{theorem} \label{thm:m-map}
For any $m \geq 1$, and any infinite continued fraction $[a_1,a_2,\dots ]$, 
and recalling our notational convention $r_m = r_{m,m}$, the sequence of values 
$r_{m}(a_1), r_{m}(a_1,a_2),$ $r_{m}(a_1,a_2,a_3), \dots, r_{m}(a_1,a_2,\dots, a_n), \dots$ converges to a real number.
\end{theorem}

\begin{proof}
    Let $X_n = \Lambda^{(m)}(a_1) \Lambda^{(m)}(a_2) \cdots \Lambda^{(m)}(a_n)$, and let $C_n$ be the cone which is the positive span of the columns of $X_n$
    (i.e. the image under $X_n$ of the positive orthant).    Let $X_n(i)$ be the $i^\mathrm{th}$ column of $X_n$. Since all entries of $\Lambda^{(m)}(a_{n+1})$
    are non-negative, the columns of  $X_{n+1}$ are non-negative linear combinations of the $X_n(i)$'s, and hence $C_{n+1} \subset C_n$. We therefore
    have a decreasing chain of sets $C_1 \supset C_2 \supset C_3 \supset \cdots$.

    Let $T_n$ be the intersection of $C_n$ with the {hyperplane defined by $x_{m+1}=1$ (i.e. working in the affine chart where the last coordinate is 1)}, 
    and let $T_n(i)$ be the scalar multiple of $X_n(i)$ with $x_{m+1}=1$. 
    Note that $T_n$ is a simplex, and by Theorem~\ref{thm:r_m} the first vertex is always
    \[ T_n(1) = \mathrm{CF}_m(a_1,\dots,a_n) = \begin{pmatrix} r_m(a_1,\dots,a_n) \\ r_{m-1,m}(a_1,\dots,a_n) \\ \vdots \\ r_{2,m}(a_1,\dots,a_n) \\ r_{1,m}(a_1,\dots,a_n) \\ 1 \end{pmatrix}. \] 
    From the previous paragraph, we have a decreasing chain of simplices $T_1 \supset T_2 \supset T_3 \supset \cdots$.
    We want to see that the intersection of this decreasing chain of $T_n$'s is a single point. Then the vectors $(r_{m}, r_{m-1,m}, \dots, r_{2,m}, r_{1,m},1)$ 
    will converge to this point.

    Note that the \emph{diameter} of $T_n$ (the maximal distance between any two points) is the length of the longest of its $(m+1)$ sides.
    If we can show that the diameters converge to zero, then we will be done.

    Note that $T_{n+1}(m+1) = T_n(1)$, so each simplex shares one vertex with the previous one. Since $X_{n+1}(m)$ is a linear combination of $X_n(1)$ and $X_n(2)$,
    the vertex $T_{n+1}(m)$ is on the line segment between $T_n(1)$ and $T_n(2)$. Moreover, it is closer to $T_n(1)$ than to $T_n(2)$.
    This is because $T_{n+1}(m) = \frac{1}{a_{n+1} (X_n)_{m+1,1} + (X_n)_{m+1,2}} (a_{n+1} (X_n)_{m+1,1} T_n(1) + (X_n)_{m+1,2} T_n(2))$ and by the lemma, we have $a_{n+1} (X_n)_{m+1,1} \geq (X_n)_{m+1,2}$. 
    Similarly, the vector $X_{n+1}(m-1)$ is a linear combination of $X_n(1), X_n(2),$ and $X_n(3)$ 
    (actually a weighted average; i.e. the coefficients sum to 1), and it is 
    most heavily weighted towards $T_n(1)$, and least towards $T_n(3)$ (again by the lemma). This is illustrated in \Cref{fig:convergence}.  
    This pattern continues where for $0\leq k \leq m+1$, the vertex $T_{n+1}(m+1-k)$ 
    is contained in the $k$-simplex defined by the points $T_n(1), T_n(2), \dots, T_n(k+1)$.

    The fact that all the vertices of $T_{n+1}$ are weighted averages of the vertices of $T_n$, and that they are most heavily weighted towards $T_1$,
    imply that $T_{n+1}$ is contained in one of the simplices of the barycentric subdivision of $T_n$. Therefore the diameter of $T_{n+1}$ is bounded above
    by the diameter of the simplices in the barycentric subdivision of $T_n$. The diameters of iterated barycentric divisions converge to zero (see e.g. \cite{hatcher}, pg 119-120),
    and so the diameters of the $T_n$'s also converge to zero.

    \begin {figure}
    \centering
    \begin {tikzpicture}
        \draw (0,0) -- (3,0) -- (2,3) -- cycle;
        \draw (0,0) -- (1.5,0) -- (1.6,1) -- cycle;

        \draw(0,0) node[left] {$T_n(1) = T_{n+1}(3)$};
        \draw(3,0) node[right] {$T_n(2)$};
        \draw(2,3) node[above] {$T_n(3)$};
        \draw(1.5,0) node[below] {$T_{n+1}(2)$};
        \draw(1.65,1) node[above] {$T_{n+1}(1)$};
    \end {tikzpicture}
    \caption {Illustration of the proof of Theorem \ref{thm:m-map} for $m=2$}
    \label{fig:convergence}
    \end {figure}
\end{proof}
	
Since any real number is equal to an infinite continued fraction, a consequence of Theorem \ref{thm:m-map} is the definition of a map from $\mathbb{R} \to \mathbb{R}$ by mapping the value $\mathbb{R} \ni x = [a_1,a_2,\dots ]$ to the limiting value $\lim_{n\to \infty} 
r_{m,m}(a_1, a_2, \dots, a_n)$.  We use this map to define higher dimensional continued fractions and include several results about extensions of such fundamental arithmetic objects.

\begin{definition} \label{def:r_mReal}
    For $m\geq 1$, and $\alpha \in \mathbb{R}$, let $[a_1,a_2,\dots ]$ be the infinite continued fraction expansion of $\alpha$.  
    Then we define $r_m(\alpha) \in \mathbb{R}$ to be the limiting value $\lim_{n\to \infty} r_{m,m}(a_1, a_2, \dots, a_n)$.
\end{definition}

Returning to Example \ref{ex:straight_double_dimers} using the notation of Theorem \ref{thm:m-map},
this example considered the infinite sequence of $r_{2,2}([a_1,a_2,\dots, a_k])$'s where $k=1,2,3,\dots$ and all $a_i=1$.  The values of 
$r_{2,2}([1])$, $r_{2,2}([1,1])$, $r_{2,2}([1,1,1]), \dots$ are thus as given by the generating function of Equation (\ref{Eq:13614}).
Considering the consecutive ratios $\frac{3}{1}, \frac{6}{3}, \frac{14}{6}, \frac{31}{14}, \dots$,
one will notice that they converge to the number $4\cos^2(\pi/7)-1 \approx 2.247 \dots$.  
This is a variation of the well-known fact
mentioned above in Example \ref{ex:fibonacci} that the number of perfect matchings of $\mathcal{G}_n$ are the Fibonacci numbers and that
the ratios $\frac{f_n}{f_{n-1}}$ converge to $\varphi = \frac{1+\sqrt{5}}{2} = 2\cos(\pi/5)$\footnote{Thus in the notation of Definition \ref{def:r_mReal}, $r_2(2\cos(\pi/5)) = 4\cos^2(\pi/7)-1$.}.  We have the following generalization of this.

\begin {theorem}
    Let $\mathcal{G}_n$ be the straight snake graph with $n$ tiles, and $\Omega_m(\mathcal{G}_n)$ the set of $m$-dimer covers.
    Let $\ell_m$ be the length of the longest diagonal in a regular $(2m+3)$-gon with side lengths $1$. Then
    \[ r_m(1,1,1,\ldots)= \lim_{n \to \infty} \frac{\left| \Omega_m(\mathcal{G}_n) \right|}{\left| \Omega_m(\mathcal{G}_{n-1}) \right|} = \ell_m \]
\end {theorem}
\begin {proof}
    The statement is equivalent to \cite[Theorem 5]{raney}, after noting that the sequence called $\phi_d(i)$ is the same as the
    left-most column of the matrix power $\Lambda^{(m)}(1)^d$, whose entries count $m$-dimer covers by Theorem \ref{thm:matrix}.
\end {proof}

\begin {remark} \label{rem:heptagons}
    It is well-known that if the diagonal lengths in an $M$-gon are
    $d_1, d_2, d_3, \dots$, then in general $d_i = U_i(\cos(\pi/M))$, where $U_k(x)$ are the Chebyshev polynomials of the second kind.
    The longest diagonal is $d_m$ where $m = \lfloor \frac{M}{2} - 1 \rfloor$, and so in the case $M=2m+3$, the theorem can be re-phrased as saying 
    \[ 
       r_m(2\cos(\pi/5)) = r_{m,m}([1,1,\dots]) =  \lim_{n \to \infty} \frac{\left| \Omega_m(\mathcal{G}_n) \right|}{\left| \Omega_m(\mathcal{G}_{n-1}) \right|} 
        = U_m \left(\cos \left( \frac{\pi}{2m+3} \right) \right) = \frac{\sin \left( \frac{m+1}{2m+3} \, \pi \right)}{\sin \left( \frac{\pi}{2m+3} \right)} 
    \]
        See also \cite{anderson20fibonacci, GoldenFields}.  The first several values are:
    \[
        \begin {array}{r|l}
            m & \ell_m \\ \hline
            1 & 1.61803\dots \\
            2 & 2.24697\dots \\
            3 & 2.87938\dots \\
            4 & 3.51333\dots \\
            5 & 4.14811\dots \\
            6 & 4.78338\dots
        \end {array}
    \]
    
    We similarly obtain that $r_{i,m}$ is the diagonal length $d_i$ in a $(2m+3)$-gon.    
    For example, $r_{1,2}([1,1,\dots]) = 1.80194\dots$,  
    $r_{1,3}([1,1,\dots]) = 1.87939\dots$, $r_{2,3}([1,1,\dots]) = 2.53209\dots$.
\end {remark}

We recall the classical result that any infinite periodic continued fraction converges to a quadratic irrational, 
i.e. an algebraic number with a quadratic minimal polynomial.  
In fact, there is a stronger converse which states that any quadratic irrational is equal to an \emph{eventually periodic} infinite continued fraction.

The following result is a generalization of this first result.  

\begin {theorem} \label{thm:eventually_periodic}
    Consider an infinite eventually periodic continued fraction given as $x = [a_1,\dots,a_n, \overline{b_1,\dots,b_k}]$,
    and let $x_n$ be the convergents of $x$. Then the limiting values $r_{i,m}(x) = \lim_{n \to \infty} r_{i,m}(x_n)$
    belong to a degree $(m+1)$ field extension of $\Bbb{Q}$. 
\end {theorem}
\begin{proof}
We suppose that the continued fraction for $x$ has period $k$ and 
let $y = [\overline{b_1,\dots,b_k}]$ be the repeating part of the continued fraction.
Also let $X$ be the matrix product $X = \Lambda^{(m)}(b_1)\Lambda^{(m)}(b_2) \cdots \Lambda^{(m)}(b_k)$.
Note that the sequence $\frac{(X^n)_{m+1-i,1}}{(X^n)_{m+1,1}}$ is a subsequence of the
one used to define $r_{i,m}(y)$. Since the latter sequence converges (by Theorem \ref{thm:m-map}),
then the former converges to the same limit. We therefore have $r_{i,m}(y) = \lim_{n\to \infty} \frac{(X^n)_{m+1-i,1}}{(X^n)_{m+1,1}}$.
It thus is sufficient to consider the form of the entries of $\lim_{n \to \infty} \frac{X^n}{(X^n)_{m+1,1}}$, 
and note that this is equivalent to identifying the eigenvectors of the matrix $X$. This is the idea behind the standard \emph{power iteration}
algorithm for finding eigenvectors and eigenvalues, and the first column of $X^n$ will converge to the direction of an eigenvector
corresponding to the largest eigenvalue of $X$. In particular, the vector $\mathrm{CF}_m(y) = (r_{m,m}(y), r_{m-1,m}(y), \dots, r_{1,m}(y), 1)^\top$,
which is precisely the first column of $\lim_{n \to \infty} \frac{X^n}{(X^n)_{m+1,1}}$, will be an eigenvector of $X$.

\medskip

Since $X$ is an $(m+1)$-by-$(m+1)$ matrix with integer entries, its characteristic polynomial, $\det(X-\lambda I)$, 
is a polynomial of degree $m+1$ with integer coefficients. Let $\lambda_1$ be the largest eigenvalue. Since the matrix $X - \lambda_1 I$
has entries in $\Bbb{Q}(\lambda_1)$, then the components of a solution to $(X - \lambda_1 I) {\bf v} = 0$ will also belong to the same field.
As noted above, $\mathrm{CF}_m(y)$ is a solution, and so $r_{i,m}(y)$ belong to this field.

So far, we have shown that $r_{i,m}(y)$ are algebraic numbers when $y$ has purely periodic continued fraction. Since $x$ is obtained from $y$ by appending the finite string
$a_1,a_2,\dots,a_n$ to the beginning of the continued fraction, this means that (up to an integer scalar multiple) $\mathrm{CF}_m(x) = \Lambda^{(m)}(a_1) \cdots \Lambda^{(m)}(a_n) \mathrm{CF}_m(y)$.
Therefore the values of $r_{i,m}(x)$ are rational linear combinations of the $r_{i,m}(y)$ values, and hence belong to the same algebraic field extension.
\end{proof}

\subsection{Ternary Continued Fractions}
\label{sec:ternary}

We specialize to the case of $m=2$ in this section and consider the values
$\lim_{n\to \infty} 
r_{1,2}(a_1, a_2, \dots, a_n)$ and $\lim_{n\to \infty} 
r_2(a_1, a_2, \dots, a_n)$ for the case of
infinite periodic continued fractions.

As we mentioned above, any infinite periodic continued fraction converges to a quadratic irrational, i.e. an algebraic number with a quadratic minimal polynomial.  
A problem going back to Hermite \cite{hermite1850extraits} asks if one can construct analogous approximations for ternary irrationals, 
i.e. algebraic numbers with cubic minimal polynomials.  
While we are not able to solve this classical two hundred year old problem, we note that applying Theorem \ref{thm:eventually_periodic} 
in the special case $m=2$ gives a partial result in this direction.

\begin{corollary} 
    For an eventually periodic continued fraction $x = [a_1,\dots,a_n, \overline{b_1,\dots,b_k}]$,
    the numbers $r_{1,2}(x)$ and $r_2(x)$ belong to a cubic extension of $\Bbb{Q}$.
\end{corollary}

\begin{example}
In the special case of a $1$-periodic continued fraction, $x = [a,a,\dots]$ is a metallic mean, i.e. $x = a + \frac{1}{x}$, and thus $x = \frac{a + \sqrt{a^2 + 4}}{2}$.  

For this case, we consider the matrix $X = \Lambda^{(2)}(a)$
and ${\bf v} = \left[\begin{matrix} \alpha \\ \beta \\ 1 \end{matrix}\right]$ is the eigenvector corresponding to the largest eigenvalue $\lambda_1$.  
But because of the especially simple form of $X = \Lambda^{(2)}(a)
 = \left[\begin{matrix} 
{a+1 \choose 2} & a & 1 \\ a & 1 & 0 \\ 1 & 0 & 0\end{matrix}\right]$, solving $X {\bf v} = \lambda_1 {\bf v}$ yields $\alpha = \lambda_1$.  
Thus for this special case, it suffices to compute the characteristic polynomial for $X= \Lambda^{(2)}(a)$,
and this yields the cubic minimal polynomial for $r_2(x)$.

We then obtain the statement that $r_2(x)$ satisfies the characteristic polynomial 
\[ \lambda^3 - \left( {a+1 \choose 2} + 1\right) \lambda^2 - \left({a \choose 2} + 1\right) \lambda + 1, \] 
and is in fact the largest root of this cubic.  

When $a=1$, we get the characteristic polynomial $\lambda^3 - 2\lambda^2 - \lambda + 1$, which is indeed the minimal polynomial 
for $4\cos^2(\pi/7)-1$ which arose in Remark \ref{rem:heptagons} for $m=2$, compare with \cite{GoldenFields}.

\end{example}

\begin{remark}
Others have studied ternary continued fractions using Jacobi's algorithm with an eye towards solving Hermite's problem of understanding periodic 
infinite ternary continued fractions \cite{daus1922normal, Bernstein+1964+137+146, brentjes1981multi, murru2013hermite,karpenkov2021hermite}.  
We note that our construction differs from those since our matrices giving rise to the relevant defining recurrence involve multichoose coefficients while the cited papers do not.  
However, as we illustrated in Remark \ref{rem:heptagons}, our construction yields geometric lengths in polygons for the fundamental 
infinite periodic continued fraction $[1,1,1,\dots ]$.
\end{remark}

\section{Future Directions and Open Questions} \label{sec:future_directions}

As was mentioned briefly in \Cref{sec:fixing_m}, when $m=1$ the rank generating functions for the poset structure on $\Omega_1(\mathcal{G})$ and $\mathcal{L}_1(\mathcal{G}^*)$
appear as the numerators of the $q$-rational numbers from \cite{mgo_20}. Morier-Genoud and Ovsienko give formulas for these polynomials in terms of $q$-deformations
of the $2 \times 2$ matrix products $\Lambda^{(1)}(a_1) \Lambda^{(1)}(a_2) \cdots \Lambda^{(1)}(a_n)$. This begs the following natural question.

\begin {ques} \label{ques:1}
    Are there $q$-analogues of Theorem \ref{thm intro1} and Theorem \ref{thm:intro3}? 
    That is, are there matrices with entries in $\Bbb{Z}[q^\pm]$, which specialize to $R$, $L$, $W$, and $\Lambda^{(m)}(a)$ when $q=1$,
    such that the ratios of entries in the corresponding matrix products are equal to ratios of the rank generating functions for the posets $\Omega_m(\mathcal{G})$?
    This would also provide $q$-analogues of the maps $r_{i,m}$ which generalize Morier-Genoud and Ovsienko's $q$-deformed continued fractions.
\end {ques}

Morier-Genound--Ovsienko \cite{mgo_20} also conjectured and O\u guz-Ravichandran \cite{ouguz2023rank} proved that the rank generating function of dimer covers of $\calG$ (i.e. $U_{P(\calG^*);1}(q)$) is unimodal. We conjecture that the same is true for higher $m$.

\begin{conj}
	The generating function $U_{P(\calG^*);m}(q)$ is unimodal.
\end{conj}

{In addition, as inspired by the partial results of \Cref{sec:ternary}, we wish to further study the map $x \to r_m(x)$ for $x\geq 1$  and $m \geq 2$.  In particular, we saw above that for $m=2$,} this map has the special property that it maps quadratic irrationals to cubic irrationals.    
The graph of $r_2$ is shown in \Cref{fig:r2_graph}. The apparent qualitative behavior leads us to the following conjecture.

\begin{conj} \label{ques:3}
    For all values $m \geq 2$ and for $x\geq 1$, the map $x \to r_m(x)$ defined above is a continuous and monotonically increasing function.   
\end{conj}

We anticipate similar behavior for $m\geq 3$, and given continuity and monotonicity, an inverse map may also be possible to define.
This conjecture will be further investigated in a forthcoming paper \cite{bos}.

\begin {ques}
    Can one explicitly construct an inverse to the map $r_m \colon \Bbb{R}_{\geq 1} \to \Bbb{R}_{\geq 1}$? In other words, given $x \in \Bbb{R}_{\geq 1}$,
    what is the integer sequence $a_1,a_2,\dots,$ such that $x = r_m(a_1,a_2,\dots)$?
\end {ques}

In particular, when $m=2$ and $x$ is a cubic irrational, an answer to this question would provide a solution to Hermite's problem.  This leads us to the following two weakenings of this question that may be more tractable but still quite challenging and interesting.

\begin {ques}
    Focusing only on continued fractions of finite length, can one explicitly construct an inverse to the map $r_m \colon \Bbb{Q}_{\geq 1} \to \Bbb{Q}_{\geq 1}$? 
    In other words, given $x \in \Bbb{Q}_{\geq 1}$, what is the finite integer sequence $a_1,a_2,\dots, a_n$ such that $x = r_m(a_1,a_2,\dots, a_n)$?
\end {ques}

\begin {ques}
In an attempt to pose an even more tractable question, consider simply the inverse images of the integers, i.e.
what is $r_m^{-1}(\Bbb{Z}_{\geq 1})$?
\end {ques}

In the special case that $m=2$, it is clear that $r_m(k) = {k+1 \choose 2}$ for each positive integer $k$, so the inverse images of triangular numbers are clear.  But if one asks for inverse images of other integers, it is not obvious.  For instance here are the inverse images for the values $1,2,\dots, 10$.

\begin{center}$\begin {array}{r|l}
    x \hspace{5em}& r_2(x) \\ \hline
    1 = 1.00000\dots & 1 \\
    3/2 = 1.50000\dots & 2 \\
    2 = 2.00000\dots & 3 \\
    16/7 = 2.28517\dots & 4 \\
     13/5 = 2.40000\dots& 5 \\
    3 = 3.00000\dots & 6 \\
    118/37 = 3.18919\dots & 7 \\
     41/12 = 3.41666\dots & 8 \\
     11/3 = 3.66666\dots & 9 \\
    4 = 4.00000\dots & 10 \\
\end {array}$
\end{center}

\section*{Postscript} 
Since posting our preprint on the {\tt arXiv} and submitting our paper for publication, 
Question \ref{ques:1} was solved in \cite{burcroff2026higher} and Question \ref{ques:3} was solved in \cite{bos}.

\section*{Ackownledgements}
We would like to thank the organizers of the Open Problems in Algebraic Combinatorics conference at the University of Minnesota in 2022, where this research project originated.
Gregg Musiker was supported by NSF grants DMS-1745638 and DMS-1854162. 
Nick Ovenhouse was supported by Simons Foundation grant 327929. 
Ralf Schiffler was supported by the NSF grant DMS-2054561 and DMS-2348909. 
Sylvester Zhang was supported by the NSF grants DMS-1949896 and DMS-1745638. 
Gregg Musiker thanks Tewodros Amdeberhan and Joe Buhler for their careful reading of an earlier draft and subsequent engaging exchanges.  
Nick Ovenhouse would like to thank Dan Douglas for helpful conversations. 
Sylvester Zhang would like to thank Joseph Pappe and Daping Weng for interesting discussions. 
The authors would also like to thank Richard Stanley for informing us about the quasi-symmetric function (Definition \ref{def:quasi}) 
and relevant references, and the anonymous referees for their helpful feedback.

\bibliographystyle{amsplain}
\bibliography{main}

\providecommand{\bysame}{\leavevmode\hbox to3em{\hrulefill}\thinspace}
\providecommand{\MR}{\relax\ifhmode\unskip\space\fi MR }
\providecommand{\MRhref}[2]{%
  \href{http://www.ams.org/mathscinet-getitem?mr=#1}{#2}
}
\providecommand{\href}[2]{#2}
\begin{thebibliography}{10}

\bibitem{anderson20fibonacci}
Stuart~D Anderson and Dani Novak, \emph{Fibonacci vector sequences and regular
  polygons}, Fibonacci Quarterly \textbf{20} (2009).

\bibitem{bos}
Etan Basser, Nicholas Ovenhouse, and Anuj Sakarda, \emph{Some aspects of higher
  continued fractions}, arXiv preprint arXiv:2401.12859 (2024).

\bibitem{bms}
V{\'e}ronique Bazier-Matte and Ralf Schiffler, \emph{Knot theory and cluster
  algebras}, Advances in Mathematics \textbf{408} (2022), 108609.

\bibitem{opac}
Christine Berkesch, Benjamin Brubaker, Gregg Musiker, Pavlo Pylyavskyy, and
  Victor Reiner, \emph{Open problems in algebraic combinatorics}, vol. 110,
  American Mathematical Society, 2024.

\bibitem{BK76}
Joel Berman and Peter K{\"o}hler, \emph{Cardinalities of finite distributive
  lattices}, Mitt. Math. Sem. Giessen \textbf{121} (1976), 103--124.

\bibitem{Bernstein+1964+137+146}
Leon Bernstein, \emph{Periodicity of {J}acobi's algorithm for a special type of
  cubic irrationals}, J. Reine Angew. Math. (1964), no.~213, 137--146.

\bibitem{brentjes1981multi}
Arne~Johan Brentjes, \emph{Multi-dimensional continued fraction algorithms}, MC
  Tracts (1981).

\bibitem{burcroff2026higher}
Amanda Burcroff, Nicholas Ovenhouse, Ralf Schiffler, and Sylvester~W Zhang,
  \emph{Higher q-continued fractions}, European Journal of Combinatorics
  \textbf{131} (2026), 104244.

\bibitem{cs_18}
{\.I}lke {\c{C}}anak{\c{c}}{\i} and Ralf Schiffler, \emph{Cluster algebras and
  continued fractions}, Compositio mathematica \textbf{154} (2018), no.~3,
  565--593.

\bibitem{claussen20}
Andrew Claussen, \emph{Expansion posets for polygon cluster algebras}, arXiv
  preprint arXiv:2005.02083 (2020).

\bibitem{daus1922normal}
Paul~H Daus, \emph{Normal ternary continued fraction expansions for the cube
  roots of integers}, American Journal of Mathematics \textbf{44} (1922),
  no.~4, 279--296.

\bibitem{hatcher}
Allen Hatcher, \emph{Algebraic topology}, Cambridge University Press, 2002.

\bibitem{hermite1850extraits}
Charles Hermite, \emph{Extraits de lettres de {M}. {C}h. {H}ermite {\`a} {M}.
  {J}acobi sur diff{\'e}rents objects de la th{\'e}orie des nombres.}, Journal
  f{\"u}r die reine und angewandte Mathematik (Crelles Journal) (1850), no.~40,
  261--278.

\bibitem{jacobi1868allgemeine}
Carl Gustav~Jacob Jacobi and Eduard Heine, \emph{Allgemeine theorie der
  kettenbruch{\"a}hnlichen algorithmen, in welchen jede zahl aus drei
  vorhergehenden gebildet wird.},  (1868).

\bibitem{karpenkov2021hermite}
Oleg Karpenkov, \emph{On {H}ermite’s problem, {J}acobi--{P}erron type
  algorithms, and {D}irichlet groups}, Acta Arithmetica \textbf{203} (2022),
  27--48.

\bibitem{klein1896representation}
F{\'e}lix Klein, \emph{Sur une repr{\'e}sentation g{\'e}om{\'e}trique du
  d{\'e}veloppement en fraction continue ordinaire}, Nouvelles annales de
  math{\'e}matiques: journal des candidats aux {\'e}coles polytechnique et
  normale \textbf{15} (1896), 327--331.

\bibitem{simplest-2-dim}
E.I. Korkina, \emph{The simplest 2-dimensional continued fraction}, Journal of
  Mathematical Sciences \textbf{82} (1996), no.~5, 3680--3685.

\bibitem{lp_08}
Thomas Lam and Pavlo Pylyavskyy, \emph{${P}$-partition products and fundamental
  quasi-symmetric function positivity}, Advances in Applied Mathematics
  \textbf{40} (2008), no.~3, 271--294.

\bibitem{llrs}
Kyungyong Lee, Li~Li, Michelle Rabideau, and Ralf Schiffler, \emph{On the
  ordering of the {M}arkov numbers}, Advances in Applied Mathematics
  \textbf{143} (2023), 102453.

\bibitem{ls_19}
Kyungyong Lee and Ralf Schiffler, \emph{Cluster algebras and {J}ones
  polynomials}, Selecta Mathematica \textbf{25} (2019), no.~4, 58.

\bibitem{mpp_18}
Alejandro~H Morales, Igor Pak, and Greta Panova, \emph{Hook formulas for skew
  shapes {I}. $q$-analogues and bijections}, Journal of Combinatorial Theory,
  Series A \textbf{154} (2018), 350--405.

\bibitem{mgo_20}
Sophie Morier-Genoud and Valentin Ovsienko, \emph{$q$-deformed rationals and
  $q$-continued fractions}, Forum of Mathematics, Sigma, vol.~8, Cambridge
  University Press, 2020.

\bibitem{murru2013hermite}
Nadir Murru, \emph{On the {H}ermite problem for cubic irrationalities}, arXiv
  preprint arXiv:1305.3285 (2013).

\bibitem{moz1}
Gregg Musiker, Nicholas Ovenhouse, and Sylvester~W. Zhang, \emph{An expansion
  formula for decorated super-{T}eichm{\"u}ller spaces}, SIGMA. Symmetry,
  Integrability and Geometry: Methods and Applications \textbf{17} (2021), 80.

\bibitem{moz2}
\bysame, \emph{Double dimer covers on snake graphs from super cluster
  expansions}, Journal of Algebra \textbf{608} (2022), 325--381.

\bibitem{moz3}
\bysame, \emph{Matrix formulae for decorated super {T}eichm{\"u}ller spaces},
  Journal of Geometry and Physics \textbf{189} (2023), 104828.

\bibitem{ms10}
Gregg Musiker and Ralf Schiffler, \emph{Cluster expansion formulas and perfect
  matchings}, Journal of Algebraic Combinatorics \textbf{32} (2010), no.~2,
  187--209.

\bibitem{msw_11}
Gregg Musiker, Ralf Schiffler, and Lauren Williams, \emph{Positivity for
  cluster algebras from surfaces}, Advances in Mathematics \textbf{227} (2011),
  no.~6, 2241--2308.

\bibitem{msw13}
\bysame, \emph{Bases for cluster algebras from surfaces}, Compositio
  Mathematica \textbf{149} (2013), no.~2, 217--263.

\bibitem{ouguz2023rank}
Ezgi~Kantarc{\i} O{\u{g}}uz and Mohan Ravichandran, \emph{Rank polynomials of
  fence posets are unimodal}, Discrete Mathematics \textbf{346} (2023), no.~2,
  113218.

\bibitem{pz_19}
R.C. Penner and Anton Zeitlin, \emph{Decorated super-{T}eichm{\"u}ller space},
  Journal of Differential Geometry \textbf{111} (2019), no.~3, 527--566.

\bibitem{propp02}
James Propp, \emph{Lattice structure for orientations of graphs}, arXiv
  preprint math/0209005 (2002).

\bibitem{propp05}
\bysame, \emph{The combinatorics of frieze patterns and {M}arkoff numbers.},
  Integers: Electronic Journal of Combinatorial Number Theory \textbf{20}
  (2020).

\bibitem{rab}
Michelle Rabideau, \emph{F-polynomial formula from continued fractions},
  Journal of Algebra \textbf{509} (2018), 467--475.

\bibitem{raney}
George Raney, \emph{Generalization of the {F}ibonacci sequence to $n$
  dimensions}, Canadian Journal of Mathematics \textbf{18} (1966), 332--349.

\bibitem{reutenauer}
Christophe Reutenauer, \emph{From {C}hristoffel words to {M}arkoff numbers},
  Oxford University Press, 2018.

\bibitem{stanley_72}
Richard Stanley, \emph{Ordered structures and partitions}, vol. 119, American
  Mathematical Soc., 1972.

\bibitem{GoldenFields}
Peter Steinbach, \emph{Golden fields: a case for the heptagon}, Mathematics
  Magazine \textbf{70} (1997), no.~1, 22--31.

\bibitem{au_15}
A~Muhammed Uluda{\u{g}} and Hakan Ayral, \emph{Jimm, a fundamental involution},
  arXiv preprint arXiv:1501.03787 (2015).

\end{thebibliography}
\end{document}